\documentclass[MNA]{mna} %

\usepackage{multirow} 
\usepackage{algorithm}
\usepackage{algpseudocode}
\usepackage{subcaption}
\usepackage{bbm}

\usepackage{enumerate}  

\SHORTTITLE{Nonparametric estimation in mean
field interacting systems of neurons}

\TITLE{Nonparametric estimation of the jump rate in mean field interacting systems of neurons} 

\AUTHORS{%
  Aline~Duarte\footnote{Departamento de Estatística, Instituto de Matem\'atica e Estat\'istica, Universidade de São Paulo, Rua do Matão 1010, 05508-090, Brazil.
    \EMAIL{alineduarte@usp.br}}
  \and 
 K\'admo~Laxa\footnote{Faculdade de Filosofia, Ciências e Letras de Ribeirão Preto, Universidade de São Paulo, Av. Bandeirantes, 3900,
Ribeirão Preto-SP, 14040-901, Brazil.
 	\EMAIL{kadmo.laxa@usp.br}}
    \and
    Eva~L\"ocherbach\footnote{ CMAP, 
Ecole Polytechnique, France.
\EMAIL{eva.loecherbach@polytechnique.edu}}
    \and 
    Dasha~Loukianova \footnote{LaMME, Universit\'e Evry-Paris-Saclay, UMR CNRS 8071, France. \EMAIL{dasha.loukianova@univ-evry.fr}}
    }

\SHORTAUTHOR{Duarte, Laxa, L\"ocherbach and Loukianova}

\KEYWORDS{spiking neurons; non-parametric estimation; kernel estimation; strong propagation of chaos; mean field limit } %

\AMSSUBJ{60F15; 60G55; 62M05; 62M20} 

\SUBMITTED{July 2, 2025} 
\ACCEPTED{July 8, 2026} 

\ARXIVID{2506.24065} 

\VOLUME{6}
\YEAR{2026}
\PAPERNUM{2}
\DOI{10.46298/mna.15979}

\ABSTRACT{We consider finite systems of $N$ interacting neurons described by non-linear Hawkes processes in a mean field frame. Neurons are described by their membrane potential. They spike randomly, at a rate depending on their potential. In between successive spikes, their membrane potential follows a deterministic flow. We estimate the spiking rate function based on the observation of the system of $N$ neurons over a fixed time interval $[0,T]$. Asymptotic results are derived in the regime where $N,$ the number of neurons, tends to infinity.  We introduce a kernel estimator of Nadaraya-Watson type and discuss its
asymptotic properties with help of the deterministic dynamical system describing the mean field limit.  We compute
the minimax rate of convergence in an  $L^2 -$error loss over a range of H\"older classes and obtain the classical rate of convergence $ N^{ - 2\beta/ ( 2 \beta + 1)} , $ where $ \beta $ is the regularity of the unknown spiking rate function.
}

\def\Pr{\mathbb{P}}

\def\eps{\varepsilon}

\newcommand\indiq{\mathbbm{1}}
\newcommand\n{\mathbb{N}}
\renewcommand\r{\mathbb{R}}
\newcommand\R{\r}
\newcommand\E{\mathbb{E}}

\newcommand\N{\n}
\newcommand{\smallO}[1]{\ensuremath{\mathop{}\mathopen{}{\scriptstyle\mathcal{O}}\mathopen{}\left(#1\right)}}

\begin{document}



\section{Introduction}
In modern neuro-mathematics, many papers are devoted to the study of large scale limits of systems of interacting neurons belonging to regions of the brain where neurons behave similarly and can be described by interactions of mean field type (see \cite{chevallier,duartechevallier,tanre1,tanre2,DGLP,SusanneEva,fournier_toy_2016,tilo2,Touboul} and the references therein). In this situation, the large population limit, also called mean field limit, provides a mesoscopic description of any single neuron by a theoretical limit dynamic which is possibly easier to study than the finite, but large system. 

In the present paper we ask the question what the mean field limit can tell us about the true finite system of neurons from a statistician's perspective. We concentrate on the question of determining how a given neuron responds to incoming stimuli, via its spiking rate function which describes the ability of the neuron to react to such stimuli, by sending sequences of action potentials. As it is commonly done, we choose to model the spiking activity of the system by a variant of a non-linear Hawkes processes, see 
\cite{Brillinger, Cessac, do, ae}. 
We also refer to \cite{GLP} and the references therein   and to \cite{Perthame,CaRou} for a PDE point of view. Our system has $N$ interacting components (the neurons) which evolve according to a deterministic flow in between successive spikes of the system and send action potentials at a rate depending on their membrane potential. Therefore, our model belongs also to the class of piecewise  deterministic Markov processes, and the Markov property is an important tool; in particular, we will rely heavily on the flow structure of the limit dynamic. 

Our goal is to estimate the unknown spiking rate function, for a fixed potential value, based on an observation of the membrane potentials of the neurons, continuously in time, during some fixed time interval $ [0, T ],$ observing more and more neurons. Our asymptotics are therefore as $ N \to \infty, $ while the time interval we consider is fixed. 

We use a classical Nadaraya–Watson kernel type estimator which is roughly defined as follows. Supposing that we want to estimate the unknown rate for any neuron having a potential value $ x^*,$  we define, for some given window size $h_N > 0, $ 
\[ 
\hat f_N ( x^*) = \frac{\# \mbox{ spikes of neurons having potential in } ]x^* - h_N, x^* + h_N [ }{\mbox{occupation time of } ]x^* - h_N, x^* + h_N [ \mbox{ during } [0, T ] \mbox{ by the whole system}},
\]
where the denominator is the global occupation time of all neurons having potentials in $]x^*-h_N;x^*+h_N[,$ 
see \eqref{eq:kernelestimator} below for the precise definition of the estimator. 

Indeed, at each time $s$ where a neuron has potential value $  \sim x^*, $  the probability that this neuron spikes during the next time step $] s, s + \Delta s ]$ is given approximately by $ f ( x^* ) \Delta s ,$ if $ \Delta s $ is small and if $f$ is the unknown spiking rate function. So if we choose a window size $h_N $ converging to $0 $ as $ N \to \infty, $ the above expression should converge, in the large population limit, to $ f ( x^* ).$  This is precisely the reason why the kernel estimator works well in general stochastic intensity models (see \cite{ApKu,DoKuk,Ku} for the intensity estimation in Poisson process models), and has been widely studied within different asymptotic scenarios such as either the observation of $N$  i.i.d. copies of the same model, as $ N \to \infty $  (see for example \cite{Vincentetthi}, Chapter 6.2 of the book \cite{Kubook} and the references therein), or the observation of a finite number of interacting processes during longer and longer time intervals, based on the ergodicity of the process, that is, as $ T \to \infty $ (see e.g. \cite{ADGP,hoffmann-olivier,Krell}, to cite just a few). 

As our model is a mean field system, our setting is closer to the first scenario, and no ergodicity is needed in our approach. Our asymptotics are based on the large population limit of the system, observing more and more neurons during a fixed time interval.  However, for each fixed $N,$ the neurons are not independent, and independence is only valid in the $ N \to \infty $ limit. More precisely, it is well-known that the \textit{propagation of chaos} property holds for such systems. This means that once we work in the hypothetical limit model, all neurons are independent. Moreover, each neuron's potential process $ (X^{N, i }_t)_t$ can be approximated by a limit process $ (x_t)_t$ which is deterministic, and we have a precise control on the rate of convergence of the strong error $ \E \left( \sup_{ s\le T }  | X_s^{N, i } - x_s| \right) , $ for any $ T > 0 $ (see Theorem \ref{theo:strongapprox} below).  

Let us now describe the major difficulty of using kernel estimators in the framework of mean field limits. To explain this difficulty,  we have to be more precise about the definition of our estimator. Instead of using indicator functions of small balls around $x^*$ (which are not continuous),  
our estimator uses a renormalized kernel function $Q_h \coloneqq \frac1h Q( \cdot/h) ,$ where $ Q $ is a fixed Lipschitz continuous kernel that integrates to $1 .$ $h$ is the window size, and we choose $h= h_N \to 0$ as $ N \to \infty. $ The estimator is then defined using $ Q_{h_N} ( x - x^* ) $ instead of the indicator of $]x^* - h_N, x^* + h_N [.$ 

In what follows, to simplify the exposition, we suppose that we want to estimate in the position $x^* = 0.$ An important step in our proofs consists then in replacing an additive functional of the type $ \int_0^t Q_{h_N} ( X^{N, i}_s) ds $ (the occupation time of a ball of radius $h_N$) by its associated limit quantity $ \int_0^t Q_{h_N} ( x_s) ds .$ A first, naive, approach would be to evaluate the error by 
\begin{multline*}
\E | \int_0^t Q_{h_N} ( X^{N, i}_s) ds - \int_0^t Q_{h_N} ( x_s) ds| \le \|Q_{h_N}\|_{Lip} \int_0^t  \E   | X_s^{N, i } - x_s|  ds\\
= h_N^{ - 2 } \|Q\|_{Lip}  \int_0^t \E   | X_s^{N, i } - x_s|  ds ,
\end{multline*}
where 
 \[
 \| Q\|_{Lip} = \sup_{ x \neq y , x, y \in \R } \frac{ | Q ( x)- Q(y) |}{ |x-y|}
\] 
denotes the Lipschitz norm of the kernel $Q$. It turns out that this upper bound is too pessimistic, since the blow up of the term $ h_N^{ - 2} $ is too strong as compared to the decay of the integrated distances $  \int_0^t \E   | X_s^{N, i } - x_s|  ds.$ 
This observation is common to all approaches relying on kernel estimators in a mean field setting. Let us cite in particular the recent article \cite{Hoffmann} on nonparametric estimation in mean field interacting diffusion models. Therein, the authors rely on Bernstein deviation inequalities, after having applied a change of measure by means of Girsanov's theorem, to cope with the above problem. This approach seems not feasible in the present setting, since the likelihood ratio process blows up as $ N \to \infty ,$ depending on more and more jump terms.  

Instead, we use a strong approximation and rely on the fact that we have the representation 
\begin{equation}\label{eq:strongapprox0} 
X_t^{N, i }= x_t + \frac{1}{\sqrt{N}} V_t^{N, i } 
\end{equation}
that allows us to represent each neuron's potential in terms of the limit potential $x_t$ and an explicit error term. Moreover, we know that for each $ p \geq 1$ and for any $T>0,$  
\[
   \E \sup_{ t \le T } | V_t^{N, i } |^p \le C_T (p) ,
\]
where the constant $C_T (p) $ does only depend on $T$ and on $p,$ but not on $N.$

The limit process $ x_t$ solves the ordinary differential equation (ODE) $ d x_t = F ( x_t) dt ,$ with $ F ( x) = b (x) + w f(x), $ where $b$ describes the leakage of the neuron's potential values in between successive spikes and where $ w$ is the limit synaptic weight. 
As a consequence, we are led to investigate expressions of the type
\[
  \int_0^t Q_{h_N} ( x_s + \frac{V^{N, i }_s}{\sqrt{N} }  ) ds .
  \]
We deal with such expressions by applying a change of variables. This change of variables uses the limit flow.  More precisely, we put  $ u\coloneqq x_s/h_N  , du = \frac1{h_N}  F ( x_s) ds = \frac{1}{h_N} F( h_N u) ds,$ such that 
\[
\int_0^t Q_{h_N} ( x_s + \frac{V^{N, i }_s}{\sqrt{N} }  ) ds = \int_{ x_0/h_N}^{x_t/h_N } Q \left( u + \frac{ V^{N, i }_{ \gamma(h_N u)}}{ h_N \sqrt{N} }\right) \frac{1}{F(h_N u)} du .
\]
Here, $ \gamma(u) $ is the inverse flow of $ s \mapsto x_s .$ This expression has then to be compared to its corresponding limit quantity (as $ N \to \infty $) which is 
\[
 \int_{x_0/h_N}^{x_t/h_N } Q(u)  \frac{1}{F(h_N u)} du , 
\]
relying only on the Lipschitz norm of the non-renormalized, original kernel $Q$ which does not depend any more on $h_N.$   

Once these techniques are settled, the remainder of the proofs is straightforward and follows standard arguments of nonparametric estimation based on kernel estimators. In particular, we obtain the classical rate of convergence $ N^{ - \frac{2 \beta}{ 1 + 2 \beta } } , $ where $ \beta $ is the (known) regularity of the spiking rate function $f,$ for the minimax error where we assess the quality of estimation by an $L^2 -$loss, uniformly within given H\"older classes that describe the regularity of the unknown spiking rate function. We also prove a central limit theorem for the rescaled estimation error, with classical speed $ \sqrt{N h_N}, $ for slightly smaller choices of $h_N = \smallO{ N^{ - \frac{1}{1 + 2 \beta }} }$. These results are gathered in our main Theorem, Theorem \ref{theo:main}.

Let us comment on a second aspect about kernel type estimators. Often their asymptotic behavior is controlled by means of a local time (for example in diffusion or branching diffusion models, see e.g. \cite{hhl}) or by means of a many-to-one formula (in some interacting particle systems, as for example in \cite{Krell}). In the present setting, 
we are able to deal with the kernel estimator without having a local time or a many-to-one formula. What makes our method work is that we are able to use a change of variables that uses the limit deterministic flow. To the best of our knowledge such an approach has never been used up to now. Let us stress at this point that our method only works if the limit ODE $ (x_s)_{ 0 \le s \le t} $ does not start from an equilibrium point; that is, we are able to estimate the spiking rate only in points $x^* $ such that $ F (x^* ) \neq 0.$   An important ingredient is the strong approximation result \eqref{eq:strongapprox0} that allows us to represent the fluctuations of the finite system around the limit dynamical system by means of an explicit error term which can be easily controlled.

We close this introduction with a short overview of recent advances in the statistical study of stochastic systems of interacting particles in a mean field frame and their associated non-linear limit processes. While such limits have been extensively studied from a probabilistic point of view
during the last three or four decades, only few statistical results are known, and mostly for mean field interacting diffusion models. We have already cited the work of Della Maestra and Hoffmann \cite{Hoffmann} above, proposing an asymptotic estimation theory in the large population setting when $N$ interacting diffusive particles evolve in a mean field frame and are observed continuously in time over a fixed time interval $[0, T ]. $ In \cite{amorino2023parameter}, the authors Amorino, Heidari, Pilipauskait{\.e} and Podolskij deal with the parameter estimation of both drift and diffusion coefficients, for discretely observed interacting diffusions, in the asymptotic framework where the step size tends to 0 and the number of particles to infinity. In \cite{semi_param_McKV}, Belomestny, Pilipauskait{\.e} and Podolskij have conducted a related study, yet in a semi parametric setting. Finally, Genon-Catalot and Laredo deal with the estimation problem in the limit model of a non-linear in the sense of McKean-Vlasov diffusion model, in the small noise regime \cite{gcatalot_laredo2}. 

The statistical theory of systems of interacting point processes in a mean-field frame is less developed, but we can mention  \cite{mf_inference_Bacry} where Bacry, Ga\"{i}ffas, Mastromatteo,  and Muzy tackle the problem of estimating the underlying parameters in systems of mean field interacting Hawkes processes, working thus in a parametric setting. Finally, let us mention \cite{mf_inference_Delattre} where Delattre and Fournier study a very particular framework: $N$ point processes are observed during a fixed time interval, interacting on an Erd\H{o}s-Renyi random graph, and the goal is to recover the underlying parameters of the linear Hawkes process as well as the connection probability.


\textbf{Organisation of the paper.}
We introduce our model and the observation scheme in Section \ref{sec:2}. Section \ref{sec:2} also contains our main result, Theorem \ref{theo:main}, which states a control of the rescaled estimation error, evaluated with respect to an $L^2 -$loss, and a weak-convergence result, with convergence rate $ \smallO{ N^{ - \frac{\beta}{ 1 + 2 \beta } }}$. Here $ \beta $ is the (known) regularity of the unknown rate function. Section \ref{sec:3} states and formalizes the strong approximation result \eqref{eq:strongapprox0} in Theorem \ref{theo:strongapprox}; the proof of this Theorem is postponed to the Appendix section.  All remaining  proofs are gathered in Section \ref{sec:proofs}. Finally, we present some simulations in Section \ref{sec:simu}, where we also discuss the influence of the long-time behavior of the limit system on the estimation quality. 

\textbf{General notation.}
Throughout this paper we shall use the following notation. For a function $ g : \R \to \R, $ we write $\| g \|_\infty = \sup_{ x \in \R } | g(x) | .$ If $g$ is Lipschitz continuous, we also use its Lipschitz norm defined by $\|g\|_{Lip}= \sup_{ x \neq y , x, y \in \R} 
\frac{ | g(x) - g(y) |}{|x-y|} .$ We use the notation $ \stackrel{\mathcal L}{\to } $ to denote convergence in distribution. 

\section{The model, assumptions and main results}\label{sec:2} 
\subsection{The model}
We consider a mean field system of $N$ interacting spiking neurons described by a generalized Hawkes process. Each neuron $i$ is characterized by its membrane potential which is a real valued stochastic process  $( X_t^{N, i }, t\geq 0 ).$
The system starts from the initial potential values $ X^{N, i }_0 , 1 \le i \le N, $ which are i.i.d. real valued random variables, distributed according to some probability measure $ \nu_0 $ having a finite first moment. The evolution of the total system is described by the following system of stochastic differential equations.
\begin{equation}\label{eq:finitesystem}
  X_t^{N, i } =  X_0^{N, i }+\int_0^t b ( X_s^{N, i } ) ds + \frac{1}{ N } \sum_{ j = 1, j \neq i  }^N \int_{[0,t]}\int_{\R_+} \int_{\R} u \indiq_{ z \le f ( X_{s-}^{N, j } ) }  \pi^j (ds, dz, du ) ,\;  1 \le i \le N. 
 \end{equation}
Here, the $ \pi^j , j \geq 1, $ are i.i.d. Poisson random measures on $ \R_+ \times \R_+ \times \R,$ having intensity $ dt dz \nu (du ) ,$ and $ \nu $ is a probability measure on $ \R.$ 
The jumps of the system are governed by these Poisson random measures: if $\pi^i$ has a jump at  time $s$ and the jump is accepted (with probability proportional to $f ( X_{s-}^{N, i } ),$ conditionally on the trajectory of the system on $[0,s[$), then each   neuron $j\neq i$ receives an additional random value $U/N$ which is added to its membrane potential $X_{s-}^{N, j } .$ Here $U\sim \nu$ is sampled independently of the trajectory of the system on $[0,s[$. The value of $X_{s-}^{N, i } $ is not changed.  Hence each neuron $j$ receives the jumps generated by the other neurons. Between two successive jump times $S<T,$   the dynamics of $X_{s-}^{N, i }$  follows the deterministic flow 
\[
X_{S+t}^{N, i }=X_{S}^{N, i }+\int_S^{S+t} b ( X_s^{N, i } ) ds,\;\; 0\leq t< T-S,
\]
for any $ 1 \le i \le N.$ 

\begin{assumption}\label{ass:1}
\begin{enumerate}
\item 
The function $b: \R \to \R $ is $ C^\infty$ and Lipschitz.   
\item The function $ f : \R \to \R_+ $ is Lipschitz continuous. 
\item
The measure $\nu$ satisfies: $ \int_\R |x|^p \nu (dx ) < \infty  $ for all $ p \geq 1 $ and  $ \nu ( \{ 0 \} ) = 0.$ We put    
\[
 w \coloneqq \int_\R u \nu ( du) .
\]

\end{enumerate}
\end{assumption}

Let $\Delta X_t^{N , i } \coloneqq X_t^{N , i } - X_{t-}^{N , i } $ be the jump height of the membrane potential of neuron $i$ at time $t$ and write $ T_1 = \inf \{ t \geq 0 : \Delta X_t^{N , i } \neq 0 \mbox{ for some } i \}, $ and 
$ T_{n+1} = \inf \{ t > T_n :  \Delta X_t^{N , i } \neq 0 \mbox{ for some } i \} ,$ $ n \geq 1,$ for the successive jump times of the system.  Here, by convention, $ \inf \emptyset = + \infty.$ 
It is well known that under the above assumption, for each fixed $N,$ equation \eqref{eq:finitesystem} possesses a unique strong solution (see Theorem~IV.9.1 of \cite{iw}) such that $ \lim_{ n \to \infty } T_n = \infty $ almost surely.

The system \eqref{eq:finitesystem} exhibits the propagation of chaos property (Theorem 8 of \cite{dfh}), and under Assumption \ref{ass:1},  in the limit $N \to \infty ,$  each neuron's potential is described by an ordinary differential equation given by 
\begin{equation}\label{eq:limitequation}
 d x_t = F ( x_t) dt,\ F( x) = b(x) + w f( x) ,\ t \geq 0.
 \end{equation}
In the sequel, the dynamical system \eqref {eq:limitequation} will play a key role in our analysis. Since $F$ is Lipschitz, for any initial value $x_0,$ it possesses a unique non-exploding solution starting from $x_0$ at time $0.$ 
In what follows, we will impose the following additional assumption.
\begin{assumption}\label{ass:1bis}
$\nu_0 = \delta_{x_0} ;$ that is, $X^{N, i }_0 = x_0 $ for all $ 1 \le i \le N.$
\end{assumption}
$x_0$ is thus the common initial value at time $t = 0 $ of all membrane potential processes $ X^{N, i}_t $ and of the limit dynamical system $x_t.$


Two examples of the evolution of the system \eqref{eq:finitesystem} with $N=5$ and $N=10$ are displayed in Figure \ref{fig:membrane_potential_train}. In these simulations, we consider the system with jump rate function 
\[
f(x)=2-e^{-x^2}
\]
and drift function 
\[
b(x)=-x.
\]
We consider $\nu \sim Uniform(-2,3)$ such that $w=0.5$. The initial membrane potential of the system is set to $x_0=-1$. In Section \ref{sec:simu} we discuss how this system can be simulated and present more simulations. Figure \ref{fig:membrane_potential_train} illustrates how the spikes of the neurons in the system modify the membrane potentials, which are initially the same for all neurons.

\begin{figure}[htbp] 
\centering
\begin{subfigure}[b]{0.49\textwidth}
    \centering
    \includegraphics[width=\textwidth]{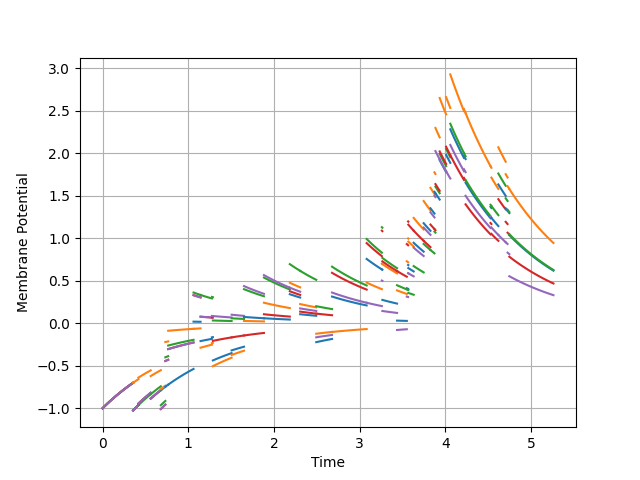}
    \includegraphics[width=\textwidth]{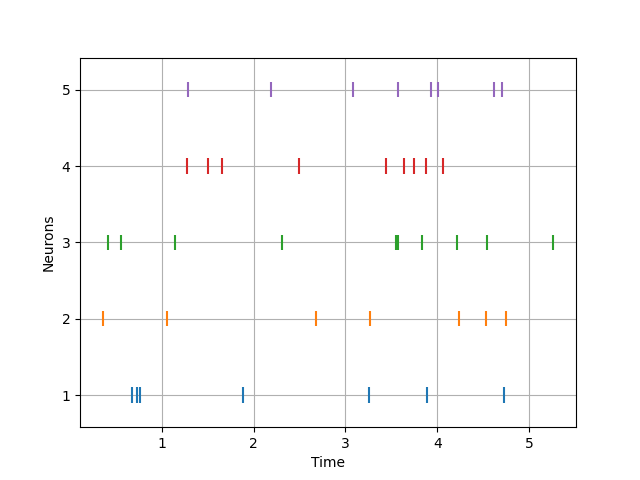}
    \caption{Evolution of the system with $N=5$.}
\end{subfigure}
\begin{subfigure}[b]{0.49\textwidth}
    \centering
    \includegraphics[width=\textwidth]{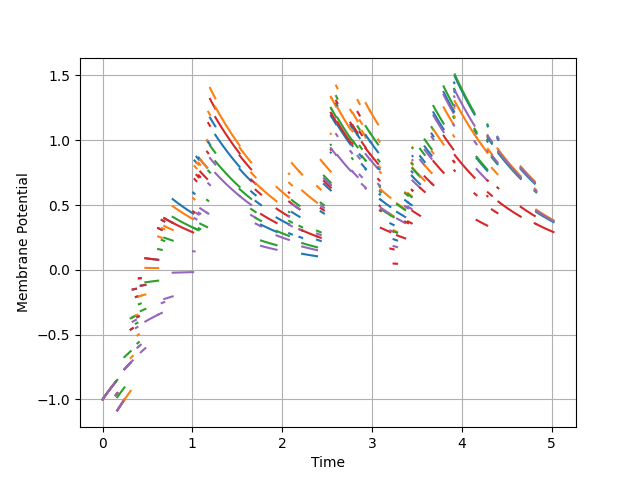}
    \includegraphics[width=\textwidth]{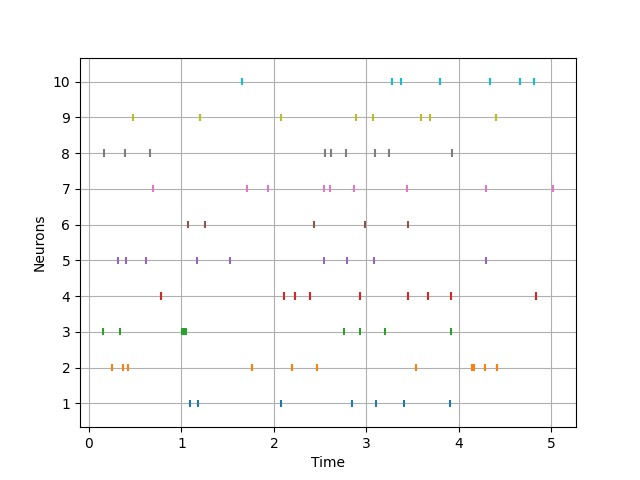}
    \caption{Evolution of the system with $N=10$.}
\end{subfigure}
\caption{Time evolution of the membrane potential of the systems with $N=5$ (left panel) and $N=10$ (right panel) and the associated spike trains.}
\label{fig:membrane_potential_train}
\end{figure}

The propagation of chaos property exhibited by the system \eqref{eq:finitesystem} as $N \to \infty$ is illustrated by Figures \ref{fig:systemconvergence1} and \ref{fig:systemconvergence2}. We consider the same system as the one considered in the simulation displayed in Figure \ref{fig:membrane_potential_train}. The only difference is that we now consider systems with $N=20$, $N=200$, $N=1000$ and $N=5000$ neurons. In Figures \ref{fig:systemconvergence1} and \ref{fig:systemconvergence2}, the time evolution of the membrane potential of $10$ neurons are displayed until time $2$ and $10$, respectively. As $N$ increases, the membrane potential of each neuron approximates a limit function, which is the solution of an ordinary differential equation describing the limit behavior of each neuron.

\begin{figure}[htbp] 
\centering
\begin{subfigure}[b]{0.49\textwidth}
    \centering
    \includegraphics[width=\textwidth]{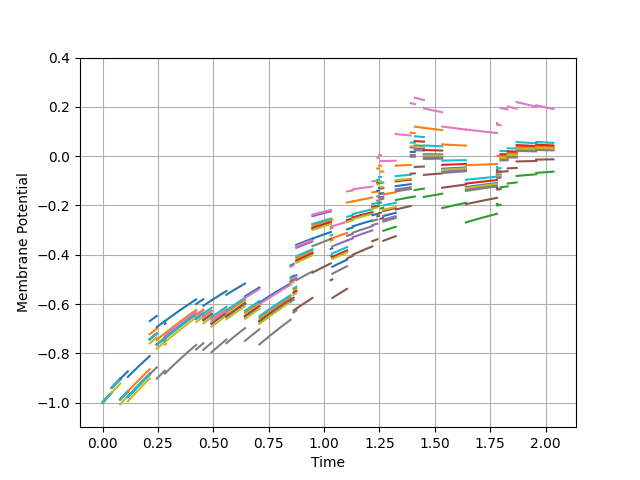}
    \caption{Evolution of $10$ neurons with $N=20$.}
\end{subfigure}
\begin{subfigure}[b]{0.49\textwidth}
    \centering
    \includegraphics[width=\textwidth]{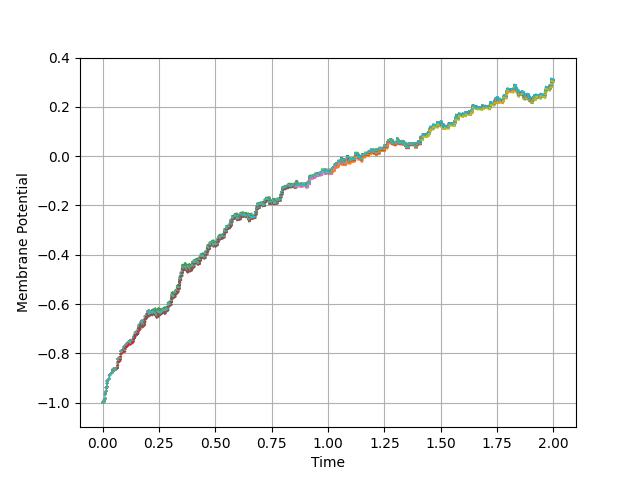}
    \caption{Evolution of $10$ neurons with $N=200$.}
\end{subfigure}
\begin{subfigure}[b]{0.49\textwidth}
    \centering
    \includegraphics[width=\textwidth]{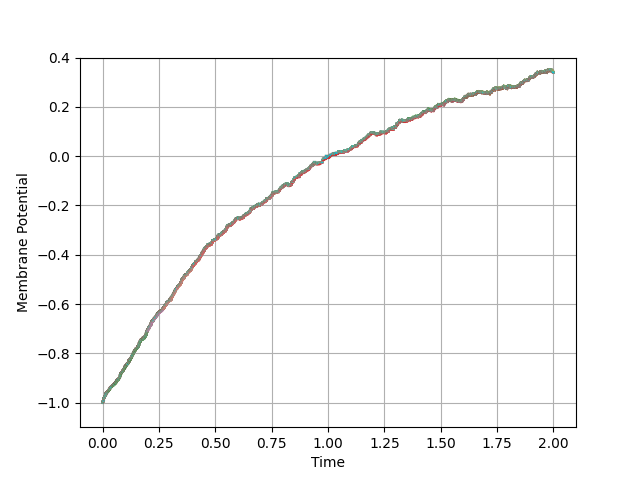}
    \caption{Evolution of $10$ neurons with $N=1000$.}
\end{subfigure}
\begin{subfigure}[b]{0.49\textwidth}
    \centering
    \includegraphics[width=\textwidth]{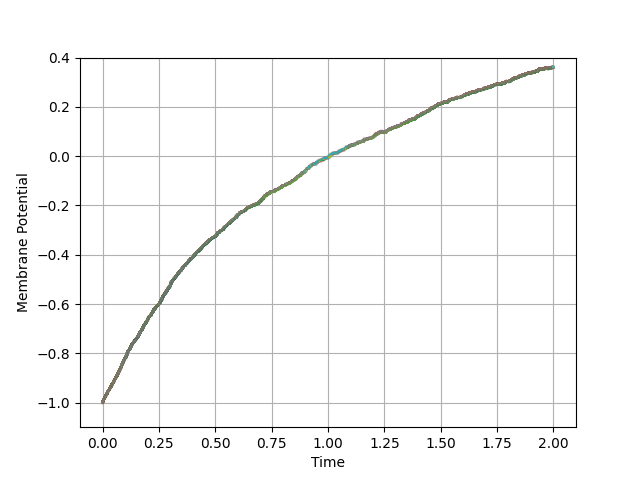}
    \caption{Evolution of $10$ neurons with $N=5000$.}
\end{subfigure}
\caption{Time evolution of the membrane potential of $10$ neurons within a system of $N$ neurons until time $2$.}
\label{fig:systemconvergence1}
\end{figure}

\begin{figure}[htbp] 
\centering
\begin{subfigure}[b]{0.49\textwidth}
    \centering
    \includegraphics[width=\textwidth]{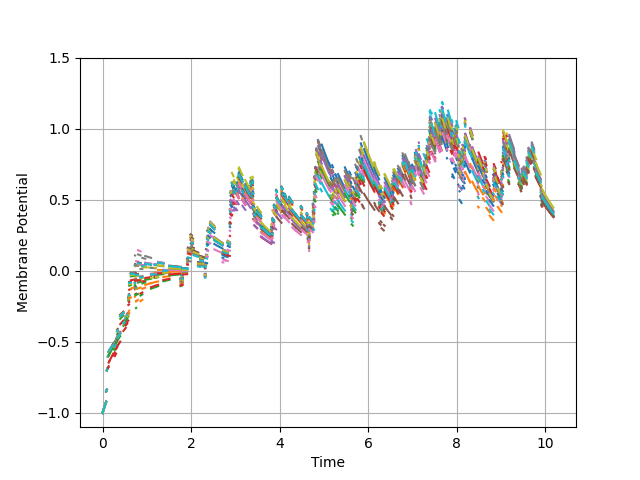}
    \caption{Evolution of $10$ neurons with $N=20$.}
\end{subfigure}
\begin{subfigure}[b]{0.49\textwidth}
    \centering
    \includegraphics[width=\textwidth]{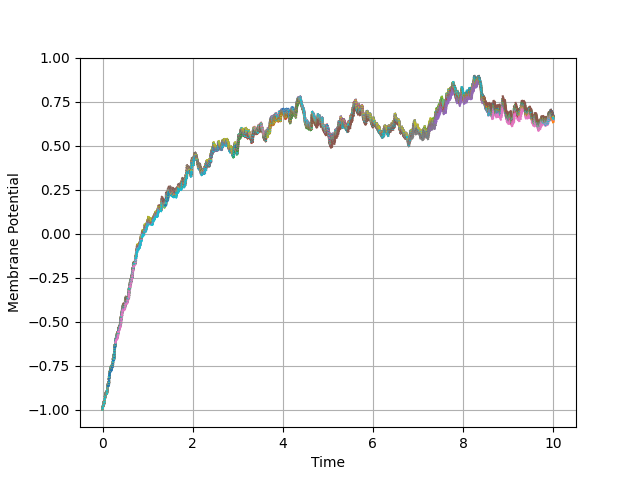}
    \caption{Evolution of $10$ neurons with $N=200$.}
\end{subfigure}
\begin{subfigure}[b]{0.49\textwidth}
    \centering
    \includegraphics[width=\textwidth]{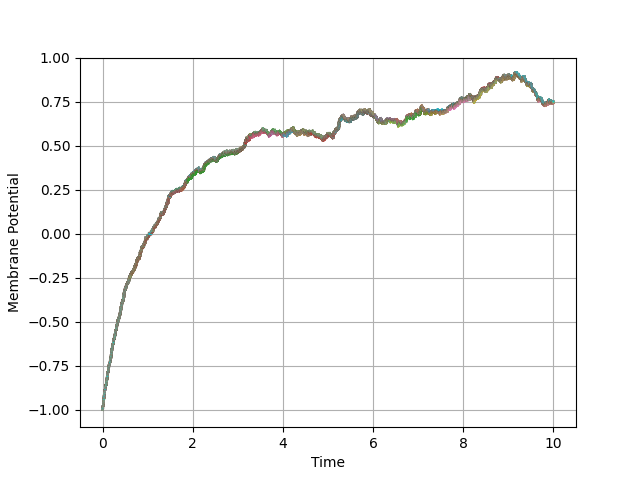}
    \caption{Evolution of $10$ neurons with $N=1000$.}
\end{subfigure}
\begin{subfigure}[b]{0.49\textwidth}
    \centering
    \includegraphics[width=\textwidth]{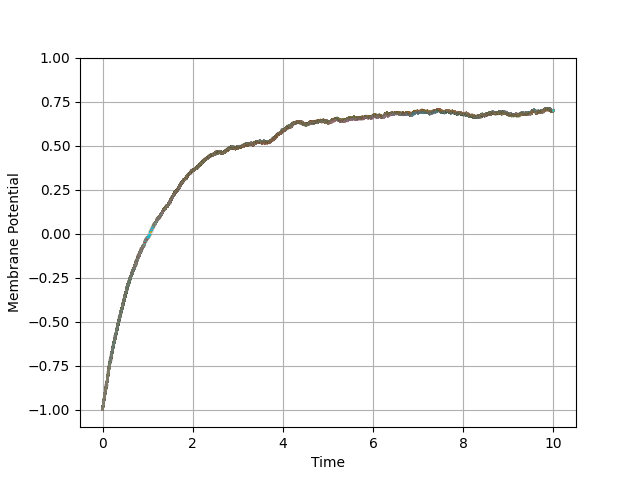}
    \caption{Evolution of $10$ neurons with $N=5000$.}
\end{subfigure}
\caption{Time evolution of the membrane potential of $10$ neurons within a system of $N$ neurons  until time $10$.}
\label{fig:systemconvergence2}
\end{figure}

\subsection{Observation scheme, definition of the estimator and main result.}
We suppose that we observe the system \eqref{eq:finitesystem} continuously in time,  during some fixed time interval $ [0, T ].  $  In particular, we suppose that we observe the initial value $x_0$ of the potentials. We wish to estimate the unknown spiking rate function $ f$ in some position $ x^* \in \R.$ Our estimation procedure works without knowing the underlying drift function $b.$

Our estimator is based on the observed jumps 
of the system. More precisely, we will use the times of the jumps (spikes) and the heights of the potential of the spiking neuron. Therefore we introduce for each $n \geq 1 $ the index $I_n \in \{ 1, \ldots, N \}$ of the spiking neuron at time $T_n.$ On the event $ 
\{ I_n = i \} ,$ it is the $i$-th neuron that spikes at time $T_n, $ and all other neurons $j$ have their potential value increased by $ u/ N, $ where $u$ is the  realisation of the random synaptic weight chosen at time $T_n.$ The precise definition of $I_n$ is as follows.  
\[
 I_n = \arg\min \{ |\Delta_n (i )| \} , \mbox{ where } \Delta_n (i)  = \Delta X_{T_n}^{N,i} = X_{T_n}^{N,i} - X_{T_n-}^{N,i} .
\]
Let us explain the above formula. Indeed, according to equation \eqref{eq:finitesystem}, if it is neuron $i$ that emits the spike at time $T_n,$ then we have that $ \Delta X^{N, i }_{T_n} = 0 $ while for all other neurons $j \neq i, $ $ \Delta X^{N, j }_{T_n} = U_n/ N, $ with $ U_n \sim \nu, U_n \neq 0.$ 

We introduce the following random measure 
\begin{equation}\label{eq:muJ}
 \mu_J^N (dt, dx) \coloneqq \sum_{ n \geq 1 } \delta_{ (T_n, X^{N, I_n}_{T_n- }) } (dt, dx) 
\end{equation} 
of the random times of the spikes and the potential value of the spiking neuron, just before the emission of the spike.   
The compensator (see \cite{JSh}, chapter II.1) of $ \mu_J^N $ is given by 
\begin{equation}\label{eq:nuJ}
 \nu_J^N (dt, dx) = \sum_{i= 1}^N  f ( X^{N, i }_{t}) dt \delta_{X_{t}^{N, i  }} (dx) .
\end{equation}  

To build our estimator, let us fix a kernel function $ Q \in C^\infty ( \R ,\R )$ which is of compact support such that $ \int_\R Q (y ) dy = 1.$ For $ h > 0 $ fixed we put 
$ Q_h (y) = \frac1h Q( y/h ) .$ Then the estimator we propose is (for a fixed $x^* \in \R $) 
\begin{equation}\label{eq:kernelestimator}
 \hat f^{N, h }_T ( x^*)  = \frac{\int_{[0, T ] \times \R } Q_h ( x-x^*) \mu^N_J ( ds, dx)}{ \int_0^T \sum_{i=1}^N  Q_h ( X_s^{N, i} - x^*) ds } ,
\end{equation}
where we define $ \frac{0}{0} \coloneqq 0.$ 

Plainly, $\hat f^{N, h }_T ( x^*) $ is the total number of spikes, during $ [0, T ], $  emitted by neurons having potential values belonging to a ball close to $ x^* $ (where the closeness is measured by means of the kernel $Q$), divided by the total occupation time of this ball close to $x^* $ by the system. 
Note that we may rewrite the above estimator in terms of the underlying Poisson random measures as 
\begin{equation}\label{eq:kernelestimator1}
\hat f^{N, h }_T ( x^*) = \frac{ \sum_{ i = 1  }^N \int_{[0,T]\times \R_+}   Q_h ( X_{s-}^{N, i }-x^*) \indiq_{\{ z \le f ( X_{s-}^{N, i } ) \}}  \bar \pi^{i} (ds, dz)}{ \sum_{i=1}^N \int_0^T  Q_h ( X_s^{N, i} - x^*) ds } ,
\end{equation}
where we write $ \bar \pi^i ( ds, dz ) \coloneqq \pi^i ( ds, dz,  \R ).$

\begin{remark}
In case that $ b (x) = - \lambda ( x - m ), $ for some $ \lambda > 0 $ and some $ m \in \R$ (exponential leakage of the potential values in between successive spikes of the system), for large $N,$ the membrane potentials of all neurons are very close. This implies that to make our estimator work, the following observation scheme is sufficient: We observe the membrane potential of a single neuron, continuously over time, and, at the same time, all spiking times of the whole system. This setting is rather realistic since neurophysiologists are nowadays able to record simultaneously the spike trains of a huge system of neurons, together with the precise membrane potential values of a small set of neurons.
\end{remark}

The quality of our estimator will depend on the asymptotic amount of time, as $ N \to \infty, $ that the system spends close to the position $ x^*$ where we wish to estimate. Recall that $ (x_s)_{ s \geq 0}  $ denotes the solution to the limit equation \eqref{eq:limitequation}. In particular, writing $ ] x_0, x_T[  $ for the interval $ \{ x \in \R : x = x_s \mbox{ for some } 0 < s < T \},$  we need to impose the following assumption. 
\begin{assumption}\label{ass:2}
$ F ( x_0 ) \neq 0$ and $x^* \in ] x_0, x_T [ . $  
\end{assumption}

Several remarks are in order concerning this assumption. 
\begin{remark}
If Assumption \ref{ass:2} is violated and $ x^* \notin ] x_0, x_T[,$ then for any choice of $ h = h_N$ such that $ h_N \to 0$ as $ N \to \infty, $ the numerator and the denominator of our estimator will converge to $0. $ This follows from the propagation of chaos property which implies that for large $N,$ the finite system will be close to the limit system. This means that in case $ x^* \notin ] x_0, x_T[,$ asymptotically, there will be no observations in a ball around $ x^*$ such that nothing can be estimated.

\end{remark}

\begin{remark}
If we suppose that $f$ is upperbounded and lowerbounded by two fixed constants  $ \underline{ f}  < \bar f  , $ we may instead study the known dynamical systems 
\[
 d l_t = ( b ( l_t) + w \underline{f})   dt \mbox{ and } d r_t = ( b (r_t) + w \bar f  ) dt ,
\]
both starting from $l_0 = r_0 = x_0 ,$  observing that $ l_s \le x_s \le r_s $ for all $s. $ 
If $ F( x_0 ) > 0, $ we have in particular $[x_0,l_T]\subset [x_0,x_T]\subset [x_0, r_T] ,$   
and both intervals $[x_0, l_T]  $ and $[ x_0 , r_T ] $ are known. If we assume $x^*\in ]x_0,l_T[ ,$ 
the condition $x^*\in ]x_0,x_T[ $ will be satisfied. Moreover, since 
\[
b(x_0)+w \underline{f} \leq F(x_0), 
\]
assuming $b(x_0)+w \underline{f} >0$ guarantees that $F(x_0) >  0,$ and hence  $F(x)\neq 0$ on $[x_0,x_T].$
The case $ F(x_0) < 0 , $ which is implied by $ b(x_0 ) + w \bar f  < 0, $ is treated analogously.
\end{remark}

To assess the limit behavior of the bias term appearing in the estimation error, we classically restrict attention to H\"older-classes of rate functions $f$ that we define now. Remember that we suppose $b$ to be known.  Fix any $ \beta \geq 1  $ and write $ \beta = k + \alpha , $  with $ k=\lfloor \beta\rfloor $  and $ 0 < \alpha \le 1. $  Then for fixed constants   $ l, L > 0, $  let 
\begin{multline}\label{eq:hb}
 H( \beta, l,  L ) = \{ f \in C^k ( \R , \R_+) :  |w f(x^*) + b (x^*)|  \geq l  , \; \sup_{ x \in \R}  f(x) \le   L , \sup_{ x \in \R}  | f' ( x) | \le L ,  \\ 
\forall x \in ]x_0, x_T[  ,   \; \forall  2 \le n \le k ,  \; | f^{( n )} (x) |\le L, \; \forall x, y \in ] x_0, x_T[ ,  | f^{( k)} (x) - f^{( k ) } (y) | \le L  | x - y |^\alpha   \}.
\end{multline}
Since to ensure the existence and uniqueness of the system \eqref{eq:finitesystem} the rate functions $f$ need to be Lipschitz, we  impose that (at least) $ \beta \geq 1.$ 

Our control of the kernel estimator \eqref{eq:kernelestimator1} relies on a precise control of the occupation time appearing in the denominator of its definition. In what follows, we write, for any $ t \le T,$ 
\begin{equation}\label{eq:at}
 A_t^{N,h} \coloneqq \frac1N \sum_{i=1}^N\int_0^t   Q_h ( X_s^{N, i} - x^*) ds \mbox{ and }   A_t^h \coloneqq  \int_0^t   Q_h ( x_s - x^*) ds.   
 \end{equation}
We will be able to give a control of the estimation error in restriction to the set 
\[
\Omega_T^{N,h,\eps}\coloneqq\left\{\left| 
\frac{A_T^{N,h}}{A_T^{h}}-1
\right|  \le \;  \eps \right\}.
\]
where $\eps < 1 $ is fixed. 

Consider the system \eqref{eq:finitesystem} starting from  $ \nu_0 = \delta_{ x_0} $ for some fixed $x_0 \in \R.$ We write $ \Pr^f_{x_0} $ for the probability measure under which the process $ X^N $ follows the equation \eqref{eq:finitesystem} with spiking rate $f,$ with initial distribution $ \nu_0 = \delta_{x_0},$ and $ \E_{x_0}^{f}$ denotes the expectation taken with respect to $ \Pr_{x_0}^f.$

\begin{theorem}\label{theo:main}
Grant Assumptions \ref{ass:1}, \ref{ass:1bis} and \ref{ass:2}.  
 Let $(h_N)_N$ be a decreasing sequence of positive numbers such that $h_N \to 0 $ and 
$ N h_N^2 \to \infty $ as $N\to\infty.$  Suppose that  $ \beta \geq  1   ,$ with $ \beta = k + \alpha , k \in \N $ and $ 0 < \alpha \le 1. $  Assume further that 
the kernel $Q$ satisfies the additional property
\[
 \int_\R u^j Q(u) du = 0 \mbox{ for all } 1 \le j \le k . 
\]
\begin{enumerate}[a)]
\item\label{itema:theomain} If $ h_N  \asymp N^{ - \frac{1}{ 2 \beta + 1 }}   $ and $r_N= N^{\frac{2\beta}{2{\beta}+1}},$ then  the kernel estimator \eqref{eq:kernelestimator} satisfies 
 \begin{equation}\label{upperbound}
  \limsup\limits_{N{\rightarrow}{\infty}}r_N \sup\limits_{{f}{\in}H( \beta, l,  L )} 
   \E_{x_0}^{f} \left(  | \hat f^{N,h_N}_T(x^*){-}{f}(x^*)|^2   \indiq_{  \Omega_T^{N,h_N,\eps} }  
\right) <   \infty  ,
 \end{equation}
and we have that 
$\lim_{N\to\infty} \sup\limits_{{f}{\in}H( \beta, l,  L )}    \Pr_{x_0}^{f}\left((\Omega_T^{N,h_N,\eps})^c \right)  =0.$ 
\item\label{itemb:theomain} 
 If $f \in H ( \beta, l , L ) $ and if $ h_N = \smallO{N^{ - \frac{1}{ 2 \beta + 1 } }}$, then, as $ N \to \infty, $ we have convergence in distribution, under the law $ \Pr_{x_0}^f , $ 
\[
 \sqrt{ N h_N}(\hat f^{N,h_N}_T ( x^*) - f(x^* ) ) \stackrel{\mathcal L}{\to} {\mathcal N} ( 0, \kappa^2), 
\]
where $ \kappa^2 = | F ( x^* )  f(x^*)| \int_{- \infty }^{\infty} Q^2 (u) du.$ 
\end{enumerate}
\end{theorem}

\begin{remark}
Let us comment on the choice of the kernel function $Q.$ Imposing that $ \int_\R  Q(u ) du = 1 $ and that $ \int_\R u^j Q(u ) du = 0 $ for all $ 1 \le j \le k, $ means that $ Q$ is a kernel of order $k$ in the sense of Definition 1.3 of Chapter 1 in \cite{Tsybakov}. This condition imposes that $Q$ cannot be a probability density. Choosing a good kernel function is important to guarantee a good rate of convergence, it allows in particular a good control of the bias term. We refer to Chapter 1.2.2 of \cite{Tsybakov} for the construction of such kernels. 
\end{remark}

\begin{remark}
We obtain the same $L^2$- rate of convergence $ \mathcal{O} (N^{ - \frac{2\beta}{ 1 + 2 \beta } })$  as for the density estimation in the classical i.i.d. setting. In particular, our rate of convergence should be optimal. A rigorous proof would rely on the  Local Asymptotic Normality  (LAN)  property for a parametric version of our statistical model (we refer to the proof of Theorem 2.2 in \cite{hhl}). The roadmap to obtain the LAN property is classical and follows ideas that can already be found in \cite{evathese}, we do not further investigate this question here.
\end{remark}

\begin{remark}[\textbf{Some remarks on partial observation schemes}]
Suppose that we are only observing a subsystem of the whole system of neurons. More precisely, let us assume that we have access to the spiking times of the first $ \gamma_N  $ neurons, where $ \gamma_N \in \N$ such that $ \gamma_N  \to \infty $ as $ N \to \infty, $ $ \gamma_N \le N.$ Then we build our estimator based on the observation of the spiking times of this subsystem. In terms of the underlying Poisson random measures, the estimator then reads as 
\begin{equation}\label{eq:kernelestimator1small}
\hat f^{N, \gamma_N, h }_T ( x^*) = \frac{ \sum_{ i = 1  }^{\gamma_N} \int_{[0,T]\times \R_+}   Q_h ( X_{s-}^{N, i }-x^*) \indiq_{\{ z \le f ( X_{s-}^{N, i } ) \}}  \bar \pi^{i} (ds, dz)}{ \int_0^T \sum_{i=1}^{\gamma_N}  Q_h ( X_s^{N, i} - x^*) ds } .
\end{equation}
The results of Theorem \ref{theo:main} remain true with the following modifications: Item~\ref{itema:theomain}) holds with $ h_N  \asymp \gamma_N^{ - \frac{1}{ 2 \beta + 1 }}$ and $r_N=\gamma_N^{\frac{2\beta}{2{\beta}+1}}.$ Item~\ref{itemb:theomain} ) holds with $h_N = \smallO{ \gamma_N^{ - \frac{1}{ 2 \beta + 1 } }} $ and with $ \sqrt{ \gamma_N h_N} $ instead of $ \sqrt{ N h_N}$.   
\end{remark}

The remainder of this article is devoted to the proof of Theorem \ref{theo:main} which follows from three preliminary results formulated in the next subsection.

\subsection{Study of the estimation error} 
We start working on the occupation time process $A_t^{N,h}$ given by \eqref {eq:at}. 
\begin{proposition}\label{prop:ANh}
If $(h_N)_N$ is such that $h_N\to 0,$  $Nh_N^2\to\infty,$ as $N\to\infty,$ then the following assertions are true. \\
\begin{enumerate}[(i)]
\item\label{item1:propanh}
\[ \lim_{N \to \infty }  \sup_{ f \in H( \beta, l, L )} \E_{x_0}^f | A_T^{N,h_N}-A_T^{h_N} |  = 0 .\]
\item\label{item2:propanh}
 $A_T^{h_N}\to(|F(x^*)|)^{-1}, $  as $N\to\infty,$  and for all $ \varepsilon > 0, $ 
\[ \lim_{N\to\infty} \sup_{ f \in H ( \beta, l, L ) } \Pr_{x_0}^f (  | A_T^{N,h_N}- (|F(x^*)|)^{-1}| \geq \varepsilon ) = 0 .\]
\item\label{item3:propanh}
$\lim_{N\to\infty} \sup\limits_{{f}{\in}H( \beta, l,  L )}    \Pr_{x_0}^{f}\left((\Omega_T^{N,h_N,\eps})^c \right)  =0.$
\end{enumerate}

 \end{proposition}
The proof of this proposition is given in Section \ref{sec:proofs}.

 Let $ \tilde \pi^i ( ds, dz ) \coloneqq \bar \pi^i ( ds, dz)-dsdz, 1 \le i \le N, $ be the compensated Poisson measures. In what follows we will frequently use the empirical measure of the finite neuron system given, for all $t\geq 0$,  by 
\[
\mu^N_t = \frac1N \sum_{i=1}^N \delta_{X^{N, i }_t } ,
\]
and for any measurable test function $g,$ we will write $ \mu^N_t ( g) = \int_\R g(x) \mu^N_t (dx).$ For any $t\geq 0$, denote by $ M_t^{N, h }$  and $ B_t^{N, h}$ the following processes.
\begin{multline}\label{eq:atn}
M_t^{N, h }: =\frac 1N\int_{[0, t ] \times \R } Q_h ( x-x^*) ( \mu^N_J - \nu_J^N ) ( ds, dx)  \\
 =\frac 1N \sum_{ i = 1  }^N \int_{[0,t]\times \R_+}   Q_h ( X_{s-}^{N, i }-x^*) \indiq_{\{ z \le f ( X_{s-}^{N, i } ) \}}\tilde \pi^{i} (ds, dz) 
\end{multline} 
and 
\begin{equation}\label{eq:atn2}
B_t^{N, h} \coloneqq \int_0^t \mu^N_s \left( Q_h ( \cdot  - x^*  )[ f( \cdot) - f( x^*) ]\right) ds=\frac1N \sum_{i=1}^N  \int_0^t Q_h ( X_s^{N, i } - x^*  )[ f( X_s^{N, i} ) - f( x^*) ] ds. 
\end{equation} 
Then we may rewrite the estimation error as 
\begin{equation}\label{eq:biais-variance}
  \hat f^{N, h }_{{T}} ( x^*) - f(x^* )  = \frac{ M_{{T}}^{N, h }}{ A_{{T}}^{N,h}} 
 + \frac{B_{{T}}^{N, h}}{ A_{{T}}^{N,h}}.
\end{equation} 

\begin{theorem}\label{theo:martingale}
Suppose that $ \nu_0 = \delta_{ x_0} $ for some fixed $x_0 \in \R$ and that Assumption \ref{ass:2} holds for this choice of $x_0.$ Let $(h_N )_N$ be such that $  h_N \to 0 $ and  
$ N h_N^2 \to \infty $ when $N\to\infty.$ 
\begin{enumerate}[(i)]
\item\label{item1:theomartingale} Suppose that $ f \in H( \beta, l, L) .$ We have convergence in law
\[
 \sqrt{ N h_N} M_{{T}}^{N, h_N }  \stackrel{\mathcal L}{\to} {\mathcal N} ( 0, \sigma^2), 
\]
where 
\[
 \sigma^2 = \frac{f(x^*)}{|F(x^*)| } \int_{- \infty }^{\infty} Q^2 (u) du.
 \]
\item\label{item2:theomartingale} Moreover, 
\[
\limsup_{N \to \infty}   \sup_{f \in H( \beta,l, L)} N h_N \E_{x_0}^f [| M_{{T}}^{N, h_N}|^2] < \infty.  
\]
\end{enumerate}
\end{theorem}

\begin{theorem}\label{theo:biais}
Grant all assumptions of Theorem \ref{theo:main}.
\begin{enumerate}[(i)]
\item \label{item1theobiais}
 If  $ h_N = \smallO{N^{ - \frac{1}{ 2 \beta + 1 } }}$,  then for all $ \varepsilon > 0, $ as $ N \to \infty ,$ 
\[
 \sup_{ f \in H ( \beta , l , L )} \Pr_{x_0}^f  (  \sqrt{N h_N} | B_{{T}}^{N, h_N } | \geq \varepsilon ) \to 0 .
\]
\item \label{item2theobiais}
 If $ h_N =  {\mathcal O} ( N^{ - \frac{1}{ 2 \beta + 1 } }) ,$ then 
\[ 
\limsup_{N \to \infty}   \sup_{f \in H( \beta,l, L)}{N h_N } \E_{x_0}^f | B_{{T}}^{N, h_N} |^2 < \infty .
\]
\end{enumerate}
\end{theorem} 
In what follows, we first present the proofs of these preliminary results. They are based on the study of the mean field limit of the finite neuron system together with a control of a quantified rate of convergence. These probabilistic results are discussed in the next section. Once the preliminary results are settled, we will then be able to conclude the proof of our main Theorem \ref{theo:main}.

\section{Strong convergence to the mean field limit}\label{sec:3}
We have already mentioned above that the system \eqref{eq:finitesystem} exhibits the propagation of chaos property  and that,  under our conditions,  in the $N \to \infty-$limit, each neuron's potential is described by an ordinary differential equation given by 
\begin{equation*}
 d x_t = F ( x_t) dt,\ F( x) = b(x) + w f( x) ,\ t \geq 0, 
 \end{equation*}
 starting from the initial value $x_0 $ at time $0.$ 
 
 Adapting a classical result obtained by Kurtz in \cite{Kurtz} to our context  implies that the following strong error bound holds, uniformly within our H\"older class of rate functions. 
 
 \begin{theorem}[cf. Theorem 2.2 of  \cite{Kurtz}]\label{theo:strongapprox}
 Suppose that $ \nu_0 = \delta_{ x_0} $ for some fixed $x_0 \in \R. $ Fix some $T > 0 .$ Then we have for any $ 1 \le i \le N$ and any $ 0 \le t \le T,$    
\begin{equation}\label{eq:strong}
X_t^{N, i }= x_t + \frac{1}{\sqrt{N}} V_t^{N, i } ,
\end{equation}
where for each $ p \geq 1,$ 
 \begin{equation}\label{eq:wasserstein0}
  \sup_{ f \in H( \beta, l, L )} \E_{x_0}^f \left( \sup_{ t \le T } | V_t^{N, i } |^p \right) \le C_T (p) ,
 \end{equation}  
for a constant $ C_T (p) $ that depends on $T$ and on $p$ but not on $N.$ 
\end{theorem}

Notice that the above result implies in particular that for any Lipschitz continuous function $g : \R \to \R$ and any $ t \le T,$  
 \begin{equation}\label{eq:wasserstein}
  \E_{x_0}^f \int_0^t |g ( X^{N, i }_s) - g( x_s) | ds  \le C_T t \|g\|_{Lip} N^{ - 1/2}.
 \end{equation}
 
For the convenience of the reader, we give the proof of Theorem \ref{theo:strongapprox} in the Appendix section.

\section{Proof of Theorem \ref{theo:martingale} and of Theorem \ref{theo:biais}}\label{sec:proofs}
We start this section with the proof of Proposition~\ref{prop:ANh}. The arguments of this proof will be used several times in the sequel. So we give all details here and then only the main ideas in the sequel. 

\begin{proof} Using definition \eqref {eq:at}, we have that 
\[
 |A_T^{N,h{_N}}-A_T^{h_N}|\leq\frac 1N\sum_{i=1}^N\int_0^T\left |Q_{h_N}(X_s^{N,i}-x^*) -Q_{h_N}(x_s-x^*)\right | ds .
\]
By Theorem \ref {theo:strongapprox}, we have that 
\[
 \frac{ X_s^{N,i}-x^*}{h_N} = \frac{x_s - x^* }{h_N} + \frac{V_s^{N, i } }{ h_N \sqrt{N}} . 
\]
In what follows we will also use the notation $ |V_T^{ * , i }| \coloneqq \sup_{ s\le T }| V_s^{N, i }|. $

Recall that $ Q$ is bounded and of compact support. Let us assume that this support is included in some interval $ [ - K , K ],$ where $ K > 0 .$  Then we can  write
\begin{multline}\label{eq:J1andJ2} 
|A_T^{N,h_N}-A_T^{h_N}|\leq \\
\frac 1N\sum_{i=1}^N\int_0^T\left |Q_{h_N}(X_s^{N,i}-x^*) -Q_{h_N}(x_s-x^*)\right | ds  \left(\indiq_{\left\{\frac{|V_T^{*,i}|}{h_N\sqrt N}\leq \frac{K}{\sqrt h_N}\right\}}+\indiq_{\left\{\frac{|V_T^{*,i}|}{h_N\sqrt N}> \frac K{\sqrt h_N}\right\}}\right ) \\ \eqqcolon J_1+J_2.
\end{multline}
Using that $Q$ is bounded  together with the fact that $Nh_N^2\to \infty ,$ by exchangeability and Markov's inequality, recalling also \eqref{eq:wasserstein0}, 
\begin{multline}\label{eq:J1}
\sup_{ f \in H ( \beta, l , L )} \E_{x_0}^f  (J_2) 
\leq 2 \|Q\|_\infty  \frac T {h_N}  \sup_{ f \in H ( \beta, l , L)}  \E_{x_0}^f \left( \indiq_{\{|V_T^{*, 1}/h_N\sqrt N|> K/\sqrt h_N\}}\right) \\
\leq 2 \|Q\|_\infty \frac{T}{h_N} \frac{ \sup_{ f \in H ( \beta, l , L )}  \E_{x_0}^f |V_T^{*,1}|^{2}}{K^{2}{N}h_N} = \frac{ \|Q\|_\infty  T  C_T ( 2) }{K^2 } ( N h_N^2)^{ -1 } \to 0, 
\end{multline}
as $ N \to \infty .$ 

To deal with $J_1,$ we will use  the  change of variables $ u = (x_s-x^*)/h_N, $ such that $x_s = x^* + h_N u . $ This change of variables is possible since, due to Assumption \ref{ass:2}, the function $t\to x(t)$ is strictly monotone  on $[0,T]$ and $F(x)\neq 0 $ for all $ x\in [x_0,x_T].$ \footnote{Indeed, by Assumption \ref{ass:2}, we know that $ x_0 \neq x_T,$ which implies that $ x_0$ cannot be a fixed point of the flow, that is, $ F ( x_0) \neq 0.$ It follows then from the properties of one dimensional ODE's that $ F ( x_t) $ is of the same sign as $ F ( x_0) $ for all $ t \in [0, T ], $ and that $ F ( x_t) \neq 0 $ for all $ 0 \le t \le T.$} 

So $ t \mapsto x_t $ is invertible. Let us write  $ \gamma (y) $ for the inverse flow of $x_s$, that is, $ x_{\gamma(y) } = y , $ such that $ s = \gamma ( x^* + h_N u ).$
Then
\begin{equation}\label{eq:15}
J_1=\frac 1N\sum_{i=1}^N\int_{\frac {x_0-x^*}{h_N}}^{\frac {x_T-x^*}{h_N}}\left| Q ( u + \frac{V^{N, i }_{ \gamma (x^* + h_N u  ) }}{ h_N \sqrt{N}} ) - Q(u) \right| \frac 1{ F ( x^* + h_N u ) } \indiq_{\left\{\frac{|V_T^{*,i}|}{h_N\sqrt N}\leq \frac K{\sqrt h_N}\right\}} du   .
\end{equation}
Here we have used that $F$ is of constant sign on $ [ x_0, x_T ]$ and that $ F$ is positive if  $ x_0 < x_T, $ negative else.

Now, recall that $|u| > K $ implies that $ Q(u)=0$ and note that
\begin{equation}\label{eq:randomsupport}
|u| >  K+\Big|\frac{V^{N, i }_{ \gamma (x^* + h_N u  ) }}{ h_N \sqrt{N}}\Big| \implies Q(u)= Q ( u + \frac{V^{N, i }_{ \gamma (x^* + h_N u  ) }}{ h_N \sqrt{N}})=0.
\end{equation}
Therefore,
\begin{multline}
J_1=
\frac 1N\sum_{i=1}^N\int_{\frac {x_0-x^*}{h_N}}^{\frac {x_T-x^*}{h_N}}  \left| Q ( u + \frac{V^{N, i }_{ \gamma (x^* + h_N u  ) }}{ h_N \sqrt{N}} ) - Q(u) \right| \\
\frac 1{ F ( x^* + h_N u ) } \indiq_{\left\{|u|\leq K+\frac{|V_{\gamma(x^*+h_Nu)}^{N,i}|}{h_N \sqrt N}\right\}} \indiq_{\left\{\frac{|V_T^{*,i}|}{h_N\sqrt N}\leq \frac K{\sqrt h_N}\right\}} du 
\leq \\ \frac{\|Q\|_{Lip}}{N}\sum_{i=1}^N\int_{\frac {x_0-x^*}{h_N}}^{\frac {x_T-x^*}{h_N}}\left| \frac {V^{N,i}_{\gamma(x^*+h_Nu)}}{h_N\sqrt N}\right|  
\frac 1{ F ( x^* + h_N u ) } \indiq_{\left\{|u|\leq K+\frac{|V_{\gamma(x^*+h_Nu)}^{N,i}|}{h_N \sqrt N}\right\}} \indiq_{\left\{\frac{|V_T^{*,i}|}{h_N\sqrt N}\leq \frac K{\sqrt{ h_N}}\right\}} du.
\end{multline} 
The following arguments will be used at several places in this article. First of all,  we suppose without loss of generality  (w.l.o.g.) that $ F$ is positive on $ [ x_0, x_T].$ 
By definition of $ H ( \beta, l, L), $ for all $ f \in H( \beta , l ,L),$ we have $ \| F' \|_\infty \le \|f'\|_\infty + \|b'\|_\infty \le L + \| b'\|_\infty $ where $ \| b' \|_\infty $ is finite by assumption. 

Recall that $ |V_T^{ * , i }| = \sup_{ s\le T }| V_s^{N, i }|. $ On the event $\left\{ \frac{|V_T^{*,i}|}{h_N\sqrt N}\leq \frac K{\sqrt h_N}\right\}$  we are considering, we have that  $ |u| \le  K+|V_{\gamma(x^*+h_Nu)}^{N,i}| / (h_N \sqrt N) \le K + K/ \sqrt{h_N} .$ So for some $ \vartheta \in (0, 1 ), $ by Taylor's formula, 
\[
F(x^* + h_N u )=F(x^*)+F'( x^* + \vartheta h_N u) h_Nu \geq F(x^*)- ( L + \|b' \|_\infty ) h_N(K+K /\sqrt h_N).
\]
If we define 
\begin{equation}\label{eq:epsilonN}
 \eps_N = ( L + \|b' \|_\infty ) h_N(K+K /\sqrt h_N) ,
\end{equation}  
which tends to $ 0$ as $ N \to \infty,$ we may conclude that 
\begin{equation}\label{eq:taylor}
\forall |u| \le K+\frac{|V_{\gamma(x^*+h_Nu)}^{N,i}|}{h_N \sqrt N},  \mbox{ we have }  F( x^* + h_N u ) \geq  F(x^*)-\eps_N  \;  \mbox{ on } \left\{ \frac{|V_T^{*,i}|}{h_N\sqrt N}\leq \frac K{\sqrt h_N}\right\} .
\end{equation}

For $ N $ sufficiently large, $ F( x^* ) - \eps_N > 0, $ such that 
\begin{multline}
J_1\leq \frac {\|Q\|_{Lip}}{F(x^*)-\eps_N}
\frac 1N\sum_{i=1}^N\int_{-\infty}^{+\infty} \left| \frac {V^{N,i}_{\gamma(x^*+h_Nu)}}{h_N\sqrt N}\right|  \indiq_{\left\{|u|\leq K+\frac{|V_{\gamma(x^*+h_Nu)}^{N,i}|}{h_N \sqrt N}\right\}}  du\leq\\
\frac {\|Q\|_{Lip}}{F(x^*)-\eps_N} \frac{2}{N}\sum_{i=1}^N\frac {|V_T^{*,i}|}{h_N\sqrt N}\left(K+\frac {|V_T^{*,i}|}{h_N\sqrt N}\right).
\end{multline}
Finally, $ | F ( x^*) | \geq l $ by definition of $ H ( \beta, l , L ) , $ such that, using \eqref{eq:wasserstein0} once more, 
\[
 \sup_{ f \in H ( \beta, l, L ) } \E_{x_0}^f ( J_1) \to 0 
\]
as $ N \to \infty,$ under the condition that $Nh_N^2\to\infty . $ All in all, we have therefore established that 
\begin{equation}
\sup_{ f \in H( \beta,l,  L )} \E_{x_0}^f |A_T^{N,h_N}-A_T^{h_N}|\to 0 
\end{equation}
as $ N \to \infty .$ 

We finally prove that $A_T^{h_N}\to |F(x^*)|^{-1}$ as $N \to \infty.$ Indeed, using once more the change of variables $u = (x_s-x^*)/h_N,$ we have that 
\[
 A_t^{h_N} = \int_{\frac {x_0-x^*}{h_N}}^{\frac {x_T-x^*}{h_N}} Q(u) \frac{1}{ F ( x^* + h_N u ) } du .
\]

Now, if $ F( x^* ) > 0, $ then necessarily $ x_0 < x^* < x_{T} , $ since $F$ is of constant sign on $ [x_0, x_T].$ This implies that  $\frac{x_0-x^*}{h_N} \to - \infty $ while $ \frac{x_T-x^*}{h_N} \to \infty . $ In the opposite case where  $ F(x^* ) < 0, $ we have  $\frac{x_0-x^*}{h_N} \to \infty $ while $ \frac{x_T-x^*}{h_N} \to - \infty .$ In both cases, since $ Q$ is of compact support,  we obtain convergence to $\frac{1}{|F(x^*)|} \int_{ - \infty }^{\infty} Q (u) du =\frac{1}{|F(x^*)|},$ since $\int Q(u) du = 1.$ 

The second part of item~\eqref{item2:propanh} and item~\eqref{item3:propanh} follow from Markov's inequality and the fact that $ |F(x^* ) | \geq  l > 0 $ uniformly within our H\"older class. 
\end{proof}

The same arguments will allow us to study now the bias term which is the most involved part of our proof.

\subsection{Proof of Theorem~\ref{theo:biais}}
Recall that the bias term is given by 
\[
 B_T^{N, h_N} =  \int_0^T \mu^N_s \left( Q_{h_N} ( \cdot  - x^*  )[ f( \cdot) - f( x^*) ]\right) ds.
\]
We introduce the associated non random limit object which is given by
\[
 B_T^{h_N} =  \int_0^T \ Q_{h_N} (x_s  - x^*  )[ f( x_s) - f( x^*) ]ds.
\]
\textbf{Step 1.}
Our first goal is to get a bound on $   |B_T^{N, h_N}  -  B_T^{ h_N} |$ and to show that this bound, even multiplied with $ \sqrt{N h_N},$ converges to $0.$ 

To do so, we use the same decomposition as in \eqref{eq:J1andJ2} and write $B_T^{N,  h_N}- B_T^{ h_N} \eqqcolon \tilde J_1 + \tilde J_2,$ where
\begin{multline*}
\tilde J_1 = \frac1N \sum_{i=1}^N \indiq_{\left\{\frac{|V_T^{*,i}|}{h_N\sqrt N}\leq \frac{K}{\sqrt h_N}\right\}}  \\
\int_0^T \left(  Q_{h_N} ( X^{N, i }_s   - x^*) [ f(  X^{N, i }_s ) - f( x^*) ]  - Q_{h_N} (x_s  - x^*  )[ f( x_s) - f( x^*) ]\right) ds ,
\end{multline*}
and 
\begin{multline*}
\tilde J_2 = \frac1N \sum_{i=1}^N \indiq_{\left\{\frac{|V_T^{*,i}|}{h_N\sqrt N} >  \frac{K}{\sqrt h_N}\right\}}  \\
\int_0^T \left(  Q_{h_N} (X^{N, i }_s    - x^*) [ f(  X^{N, i }_s  ) - f( x^*) ]  - Q_{h_N} (x_s  - x^*  )[ f( x_s) - f( x^*) ]\right) ds .
\end{multline*}

To deal with $ \tilde J_2,$ we start noticing that for any real numbers $a, b,$ uniformly for all $ f \in H( \beta, l, L),$  
\begin{equation} \label{eq:qtimesfbound}
|Q_{h_N} (a - b ) [ f(a) - f(b) ]|\leq  \frac{1}{h_N} \|Q\|_{\infty}\|f\|_{Lip}Kh_N  \le \|Q\|_\infty L K ,
\end{equation}
where we have used that $ Q$ is bounded and of compact support, with support included in $ [ - K , K ],$ for some $ K > 0 ,$ and that $f$ is Lipschitz continuous.
Applying (\ref{eq:qtimesfbound}) with $ a = X^{N, i }_s $ or $ a = x_s $ and $b = x^*, $ we deduce from this that
\begin{equation}\label{eq:upperboundtildej2}
 \tilde J_2 \le 2 \|Q\|_\infty L K  T \frac1N \sum_{i=1}^N  \indiq_{\left\{\frac{|V_T^{*,i}|}{h_N\sqrt N} >  \frac{K}{\sqrt h_N}\right\}}  ,
\end{equation}
such that, as in the proof of Proposition~\ref{prop:ANh}, 
\[
 \sup_{ f \in H( \beta, l, L) } \E_{x_0}^f ( \tilde J_2 )  \le \|Q\|_\infty L K  T \frac{C_T (2) }{K^2 h_N N } .
\]
In particular, since $ h_N N \to \infty  $ (which follows from $h_N^2 N \to \infty $ and the fact that $ h_N \to 0$), 
\begin{equation}\label{eq:tildeJ2} \lim_{ N \to \infty }  \sup_{ f \in H( \beta, l, L) }  \sqrt{N h_N} \E_{x_0}^f ( \tilde J_2 ) = 0.
\end{equation} 

To deal with $ \tilde J_1,$ we use again  the change of variables $ u = (x_s-x^*)/h_N, $ as in the proof of Proposition~\ref{prop:ANh}.  Then 
\[
 \tilde J_1 =\frac1N \sum_{i=1}^N 
(E_1^{N, i} -  E_2^{N, i} ),
\]
 where 
\begin{multline*}
 E_1^{N, i} =  \int_{\frac{x_0-x^*}{h_N}}^{\frac{x_T- x^*}{h_N}} [ Q ( u + \frac{V^{N, i }_{ \gamma (x^* + h_N u  ) }}{ h_N \sqrt{N}} ) - Q(u) ]  [ f( x^* + h_N u + \frac{ V^{N, i }_{ \gamma (x^* + h_N u  ) }}{\sqrt{N} } ) - f(x^*) ] \\
 \frac{1}{ F ( x^* + h_N u ) } du \indiq_{\left\{\frac{|V_T^{*,i}|}{h_N\sqrt N} \le   \frac{K}{\sqrt h_N}\right\}}  
\end{multline*} 
and 
\begin{multline*}
E_2^{N,i} = \\\int_{\frac{x_0-x^*}{h_N}}^{\frac{x_T- x^*}{h_N}}  Q (  u )  [  f( x^* + h_N u + \frac{ V^{N, i }_{ \gamma (x^* + h_N u  ) }}{\sqrt{N} } ) - f ( x^* + h_Nu )  ]  \frac{1}{ F ( x^* + h_N u ) } du \indiq_{\left\{\frac{|V_T^{*,i}|}{h_N\sqrt N} \le   \frac{K}{\sqrt h_N}\right\}} .
\end{multline*} 

In what follows, we suppose w.l.o.g. that $ x_0 < x_T $ such that $F$ is of positive sign on $ [x_0, x_T].$ Due to the compact support of $Q,$ the domain of integration within the integral defining $ E_1^{N, i } $  can be restricted to values of $u$ such that $|u| \le K+\Big|\frac{V^{N, i }_{ \gamma (x^* + h_N u  ) }}{ h_N \sqrt{N}}\Big| .$ Using that $Q$ is Lipschitz and that $ f$ is Lipschitz with Lipschitz constant bounded by $L, $ uniformly within our function class, and the arguments and notations of the proof of Proposition~\ref{prop:ANh} above, 
we have 
\[
| E_1^{N,i} | \le \frac{\|Q\|_{Lip} }{ F(x^* ) - \eps_N}  L  \int_{ - \infty}^\infty   \frac{|V^{*, i }_{ T}|}{ h_N \sqrt{N}}  \left( h_N |u| + \frac{|V^{*,i}_T|}{ \sqrt{N} } \right)   \indiq_{\left\{|u|\leq K+\frac{|V_T^{*,i}|}{h_N \sqrt N}\right\}}    du .
\]
Collecting all constants within a generic constant $C$ that depends only on the function class $ H ( \beta, l, L ), $ we may therefore upper bound the above expression by 
\[
| E_1^{N,i} | \le   \frac{C }{ F(x^* ) - \eps_N}  \left(    \frac{|V_T^{*,i}|}{ \sqrt{N}}   +\frac{|V_T^{*,i}|^2}{h_N N } + \frac{|V_T^{*,i}|^3}{N^{3/2}  h_N^2 } \right) .
\]

Similarly, in $ E_2^{N, i }, $ the domain of integration can be restricted to $ |u| \le  K ,$ and we obtain 
\[
  | E_2^{N,i}| \le \frac{C }{ F(x^* ) - \eps_N}   \|f\|_{Lip}  \frac{|V_T^{*,i}|}{\sqrt{N} } \le  C L \frac { |V_T^{*,i}|} {\sqrt{N}} . 
\]
The two previous bounds imply that
 \begin{multline}\label{eq:impbias}
\tilde J_1 \leq \frac 1N\sum_{i=1}^N \left( |E_1^{N,i}| + |E_2^{N,i}|\right)\\
 \leq \frac{C }{ F(x^* ) - \eps_N}  \left( \frac 1{\sqrt N}+    \frac{1}{ {h_N N}  } +   \frac{1}{ h_N^{2} N^{3/2}} \right) \frac 1N\sum_{i=1}^N (1+|V_T^{*,i}|+|V_T^{*,i}|^3)  .
 \end{multline}

In particular, under the condition $h_N^{3/2}N\to\infty,$ since $ |V_T^{*,i}| $ has all moments (recall \eqref{eq:wasserstein0}), we deduce from \eqref{eq:impbias} together with \eqref{eq:tildeJ2} that
\begin{equation}\label{eqdrift1}
 \sqrt{Nh_N}\sup_{ f \in H ( \beta,l,  L)} \E_{x_0}^f | B_t^{N,  h_N}- B_t^{ h_N}| \to 0 . 
\end{equation}
Moreover, from \eqref {eq:impbias},
\begin{equation*}
( \tilde J_1) ^2  \leq  \frac{C }{ (F(x^* ) - \eps_N)^2}  \left( \frac 1{ N}+    \frac{1}{ {h_N^2 N^2}  } +   \frac{1}{ h_N^{4} N^{3}} \right) \frac 1N\sum_{i=1}^N [(1+|V_T^{*,i}|+|V_T^{*,i}|^3)]^2  ,
 \end{equation*}
such that, recalling \eqref{eq:upperboundtildej2}, using that  $h_N^{3/2}N\to\infty$ (which follows from $h_N^2 N \to \infty $ and the fact that $ h_N \to 0$),
\begin{equation}\label{eq:impbias22}
 Nh_N \sup_{ f \in H ( \beta, l, L)} \E_{x_0}^f | B_T^{N,  h_N}- B_T^{ h_N}|^2  \to 0 
 \end{equation}
 as $ N \to \infty. $ 
  
\textbf{Step 2.}
We now investigate further the expression of $ B_T^{h_N} .$ Using the change of variables $ u = (x_s-x^*)/h_N, $ we have
\[
 B_T^{h_N}  =  \int_{\frac{x_0-x^*}{h_N}}^{\frac{x_T- x^*}{h_N}}  Q (  u ) \frac{1}{F( x^* + h_N u)} [ f ( x^* + h_Nu ) - f( x^* ) ] du .
\]

For $N$ large enough, we have  $supp Q  \subset ] \frac{x_0 - x^* }{ h_N}, \frac{x_T- x^* }{ h_N}[  ,$ 
such that the above integral equals 
\[
 B_T^{h_N}  =  \int_{\R}  Q (  u ) \frac{1}{F( x^* + h_N u)} [ f ( x^* + h_Nu ) - f( x^* ) ] du 
\]
which does not depend on $T$ any more. 
We rewrite this expression as $B_T^{h_N}  = B^{h_N, 1}  + B^{h_N, 2}  , $ where 
\[
B^{h_N, 1} =  \frac{1}{F( x^*) } \int_{\R } Q (  u )  [ f ( x^* + h_Nu ) - f( x^* ) ]  du
\]
and 
\[
 B^{h_N, 2 } =   \int_{\R } Q (  u ) [ \frac{1}{F( x^* + h_N u)}- \frac{1}{F( x^*) }]  [ f ( x^* + h_Nu ) - f( x^* ) ] du.
\]
We start developing $B^{h_N, 1} .$ By assumption, 
\[
  f ( x^* + h_Nu ) - f( x^* ) = \sum_{l=1}^k \frac{ f^{(l)} (x^*)}{ l!} h_N^l u^l + \frac{ f^{(k)} ( y ) - f^{(k) } (x^*) }{k!} (h_N u)^k, 
\]
where $ y = y(u)  \in ] x^* , x^* + h_N u [ $ if $ u > 0 $ and $y \in  ]  x^* + h_N u, x^*  [,$ if $ u < 0.$  By our assumptions on the kernel $Q,$ this implies that 
\[
 B^{h_N, 1} =  \frac{1}{F( x^*) } \int_{\R } Q (  u )  \frac{ f^{(k)} ( y ) - f^{(k) } (x^*) }{k!} (h_N u )^k  du . 
\]

Since $ y \in [x_0, x_T],$ we have that $ | f^{(k)} ( y ) - f^{(k) } (x^*) | \le L | y - x^* |^\alpha \le L h_N^\alpha |u|^\alpha , $ which implies that 
\begin{equation}\label{eq:bth1} | B^{h_N, 1}|  \le L  h_N^{\beta } \frac{1}{F( x^*) } \int_{\R } |Q (  u )|  |u|^\beta du \le \frac{ L}{ l} h_N^\beta \int_{\R } |Q (  u ) | |u|^\beta du , 
\end{equation}
where we have used that $ F ( x^* ) \geq l $ and where we recall that $ \beta = k + \alpha.$ 
Thus we have just proved that 
\[
 \sup_{ f \in H ( \beta,l,  L ) } | B_T^{h_N, 1}|  \le C h_N^\beta , 
\]
for all $ N $ sufficiently large. 

The control of $ B^{h_N, 2} $ is similar. The domain of integration for $B^{h_N, 2} $ can be restricted to $ |u| \le K $ such that for all such $ u, $ $F ( x^* + h_N u ) \geq F(x^* ) - ( \| b' \|_\infty + L ) K h_N $ (recall the argument leading to \eqref{eq:epsilonN}) which is strictly lower bounded for all $ N $ sufficiently large. Therefore $ F$ is uniformly lower bounded on the domain of integration. The derivatives of $ \frac{1}{F}$ up to order $k$ involve sums of derivatives of order $ j $ of $F,$ for all $ j \le k, $ hence of $f$ and of $b,$ divided by powers of $F.$ Since we suppose $b $ to be $ C^\infty ,$ to be known, and since we impose the uniform controls on the derivatives of $f$ on $ [x_0, x_T],$ it follows  that $  \frac{1}{ F}  $ belongs to a similar H\"older class as $f,$ of same regularity, with different constants.  Finally, developing both $[ \frac{1}{F( x^* + h_N u)}- \frac{1}{F( x^* ) }]  $ and $f ( x^* + h_Nu ) - f( x^* ),$ we then obtain that 
\begin{equation*}
 \sup_{ f \in H ( \beta, l, L ) }  | B_T^{h_N, 2}|   \le C   h_N^{\beta } 
\end{equation*}
as well, and finally
 \begin{equation}\label{eq:bth}  \sup_{ f \in H ( \beta,l,  L ) }  | B_T^{h_N}|   \le C   h_N^{\beta }  ,
\end{equation}
for all $ N $ sufficiently large. 

To conclude the proof, we write 
\begin{equation*}
 |B_T^{N, h_N}|\leq  |B_T^{N, h_N}-B_T^{ h_N}|+|B_T^{ h_N}| \leq  
  |B_T^{N, h_N}-B_T^{ h_N}|+Ch_N^\beta ,
 \end{equation*}
for all $N $ sufficiently large, and
 \begin{equation*}
 \sqrt{Nh_N}\sup_{ f \in H ( \beta,l, L)} \E_{x_0}^f |B_T^{N, h_N}|\leq   \sqrt{Nh_N}\sup_{ f \in H ( \beta, l, L)} \E_{x_0}^f |B_T^{N, h_N}-B_T^{ h_N}|+C\sqrt{Nh_N}h_N^\beta.  
 \end{equation*}
By \eqref{eqdrift1}, the first term of the above sum tends to zero supposing only that $h_N^{3/2}N\to\infty.$ To conclude the proof of item~\eqref{item1theobiais}, we need to ensure that $ \sqrt{h_N N }  h_N^\beta \to 0$ as $ N \to \infty. $ This imposes the choice $ h_N = \smallO{N^{ - \frac{1}{ 2 \beta + 1 } } }$ and concludes the proof of item~\eqref{item1theobiais}.

To prove item~\eqref{item2theobiais}, we write, for $N$ sufficiently large, 
\begin{equation*}
 |B_T^{N, h_N}|^2\leq 2 |B_T^{N, h_N}-B_T^{ h_N}|^2+2|B_T^{ h_N}|^2 \leq  
  2|B_T^{N, h_N}-B_T^{ h_N}|^2+Ch_N^{2\beta } 
\end{equation*}
and 
 \begin{equation*}
 {Nh_N}\sup_{ f \in H ( \beta, l, L)} \E_{x_0}^f |B_T^{N, h_N}|^2\leq  {Nh_N}\sup_{ f \in H ( \beta,l,  L)} \E_{x_0}^f |B_T^{N, h_N}-B_T^{ h_N}|^2+C{N}h_N^{1+2\beta}.  
 \end{equation*}
Using \eqref{eq:impbias22}, the first term of the above sum tends to zero. Under the condition that $ h_N = {\mathcal O} ( N^{ - \frac{1}{ 2 \beta + 1 } } )  ,$ we then get
 $ \limsup_{ N \to \infty} { N }  h_N^{1+2\beta}  < \infty,$ which concludes the proof of item~\eqref{item2theobiais}. $\qed $

We now turn to the proof of Theorem~\ref{theo:martingale}.

\subsection{Proof of Theorem \ref{theo:martingale}} 
Recall the definition of $  M_t^{N, h } $ in \eqref{eq:atn} and let $ W^{N, h_N} _t \coloneqq \sqrt{N h_N} M_t^{N, h_N}.$ Fix $ t_0 < T  $ such that $ x^* \in ] x_0, x_{t_0 } [  $ (recall Assumption \ref{ass:2}). Instead of evaluating $ W^{N, h_N}$ only at time $T, $ we shall consider the whole process $ (W^{N, h_N}_s)_{ s \geq t_0}$ which is a martingale (for any fixed $N$), having its jumps bounded by $ C/ \sqrt{N h_N} $ (since $Q$ is bounded) which converges to $0.$ 

For any $ s \geq t_0,$ its predictable quadratic covariation process $\langle W^N\rangle_s$is given by 
\begin{multline}\label{eq:wwt}
\langle W^{N, h_N}\rangle_s= \\ \frac{h_N}{N} \int_{ [0, s ] \times \R } Q^2_{h_N} ( x - x^* ) \nu^N_J ( dv, dx) = \frac{h_N}{N} \sum_{i=1}^N \int_0^s   f ( X^{N, i }_v ) Q^2_{h_N} ( X^{N, i }_v - x^* ) dv\\
 =\frac{1}{N} \sum_{i=1}^N \int_0^s   f ( X^{N, i }_v )  \frac{1}{h_N}Q^2 ( \frac{X^{N, i }_v - x^* }{h_N} ) dv \eqqcolon\frac{1}{N} \sum_{i=1}^N W^{N,i, h_N }_s.
\end{multline}
In what follows we show that $ \langle W^{N, h_N}\rangle_s$ converges in probability to some deterministic limit. To prove this, we study each single term $ W^{N,i, h_N}_s =h_N \int_0^s   f ( X^{N, i }_v ) Q^2_{h_N} ( X^{N, i }_v - x^* ) dv $ as before. We then compare it to its associated non random limit expression which is given by
\[
W^{h_N}_s\coloneqq h_N \int_0^s   f ( x_v ) Q^2_{h_N} ( x_v - x^* ) dv =  \int_0^s  f ( x_v ) \frac{1}{h_N}Q^2 ( \frac{x_v - x^* }{h_N} ) dv .
\]
As in the proof of Prop.~\ref{prop:ANh}, localizing according to the fact whether $ |V_T^{*,i}| / (h_N\sqrt N ) \leq  K/ \sqrt h_N $ or not, we rewrite 
\[
 \frac{1}{N} \sum_{i=1}^N ( W^{N,i, h_N }_s - W^{h_N}_s) \eqqcolon \bar J_1 + \bar J_2 .
\]
Since $f$ is bounded, applying the same arguments as those that yield \eqref{eq:J1}, we obtain 
\[
 \sup_{ f \in H ( \beta, l , L )} \E_{x_0}^f  (\bar J_2)  \to 0 
\]
as $N \to \infty .$  We now deal with 
\[
 \bar J_1 =  \frac{1}{N} \sum_{i=1}^N ( W^{N,i, h_N }_s - W^{h_N}_s) \indiq_{\left
\{\frac{|V_T^{*,i}|}{h_N\sqrt N}\leq \frac{K}{\sqrt h_N}\right\}} .
\]

Applying the change of variables $ h_N u = x_v - x^* ,$ we get  
\[
 W^{N,i,h_N}_s =  \int_{\frac{x_0- x^*}{h_N}}^{\frac{x_s-x^*}{h_N}} f ( x^* + h_N u  + \frac{V^{N, i }_{\gamma ( x^* + h_N u )}  }{ \sqrt{N}} ) Q^2 (u  + \frac{V^{N, i }_{\gamma ( x^* + h_N u )}  }{ h_N\sqrt{N}}  ) \frac{1}{F ( x^* + h_N u )} du 
\]
and
\[
W^{h_N}_s =  \int_{\frac{x_0- x^*}{h_N}}^{\frac{x_s-x^*}{h_N}}    f ( x^*+h_N u)  Q^2 (u ) \frac{1}{F ( x^* + h_N u )}du .
\]
Supposing w.l.o.g. that $ F $ is positive on $ ] x_0, x_t[$, we therefore obtain  
\begin{equation}\label{eq:errorm}   \bar J_1 \le \frac1N \sum_{i=1}^N ( E_1^{N,i} + E_2^{N,i}) ,
\end{equation} 
where 
\begin{multline*}
 E_1^{N,i} =\indiq_{\left
\{\frac{|V_T^{*,i}|}{h_N\sqrt N}\leq \frac{K}{\sqrt h_N}\right\}}   \\
\times \int_{\frac{x_0- x^*}{h_N}}^{\frac{x_s-x^*}{h_N}} f ( x^* + h_N u  + \frac{V^{N, i }_{\gamma ( x^* + h_N u )}  }{ \sqrt{N}} )\left| Q^2 (u  + \frac{V^{N, i }_{\gamma ( x^* + h_N u )}  }{ h_N\sqrt{N}}  ) - Q^2 (u    )\right| \frac{1}{F ( x^* + h_N u )} du 
\end{multline*}
and 
\begin{multline*}
 E_2^{N,i} = 
 \indiq_{\left
 	\{\frac{|V_T^{*,i}|}{h_N\sqrt N}\leq \frac{K}{\sqrt h_N}\right\}} 
 	\\ \times \int_{\frac{x_0- x^*}{h_N}}^{\frac{x_s-x^*}{h_N}} Q^2 (u) \left|f ( x^* + h_N u  + \frac{V^{N, i }_{\gamma ( x^* + h_N u )}  }{ \sqrt{N}} ) - f (x^* + h_N u )  \right| \frac{1}{F ( x^* + h_N u )} du. 
\end{multline*}
Due to the compact support of $Q,$ the domain of integration within the integral defining $ E_1^{N, i } $  can be restricted to values of $u$ such that $|u| \le K+\Big|\frac{V^{N, i }_{ \gamma (x^* + h_N u  ) }}{ h_N \sqrt{N}}\Big| .$
Using the arguments and notations of the proof of Proposition~\ref{prop:ANh} above, and recalling in particular \eqref{eq:epsilonN}, exploiting the Lipschitz continuity of $ Q^2, $ we obtain 
\begin{multline*}
 E_1^{N,i} \le \frac{L \| Q^2 \|_{Lip} }{ F( x^* ) - \eps_N} \int_{-\infty}^\infty  \indiq_{ \left\{  |u| \le K+\frac{|V^{N, i }_{ \gamma (x^* + h_N u  ) }|}{ h_N \sqrt{N}}  \right\} }   \frac{V^{N, i }_{\gamma ( x^* + h_N u )}  }{ h_N\sqrt{N}}  du \indiq_{\left
\{\frac{|V_T^{*,i}|}{h_N\sqrt N}\leq \frac{K}{\sqrt h_N}\right\}}   \\
   \le C\frac{|V_T^{*,i}|}{h_N \sqrt{N} }  ( 1 + \frac{|V_T^{*,i}|}{h_N \sqrt{N} }) ,
\end{multline*}  
where the last inequality holds for all $ N $ sufficiently large such that $ F ( x^* ) - \eps_N $ is uniformly lowerbounded. Therefore, 
\begin{equation}\label{eq:e1bound} \sup_{ f \in H ( \beta,l, L) } \E_{x_0}^f  E^{N, i }_1 \le C \sup_{ f \in H ( \beta, l,  L) } \E_{x_0}^f  \left (  \frac{|V_T^{*,i}|}{h_N \sqrt{N} } + \frac{  |V_T^{*,i}|^2 }{h_N^2 N }\right) \to 0 
\end{equation}
as $ N \to \infty , $ since 
$ h_N^2 N \to \infty $ by assumption.

Finally, it is easily seen that 
\[
E_2^{N,i} \le \frac{L}{F(x^* ) - \eps_N } \frac{|V_T^{*,i}|}{ \sqrt{N} } \int_\R Q^2 (u) du  ,
\]
due to the Lipschitz continuity of $f$ with Lipschitz constant upper bounded by $L,$ such that 
\[
 \lim_{ N \to \infty } \sup_{ f \in H ( \beta,l,  L) } \E_{x_0}^f  E^{N, i }_2 = 0 
\]
as well. This implies that 
\[
 \lim_{ N \to \infty } \sup_{ f \in H ( \beta,l, L) } \E_{x_0}^f \left( | \langle W^{N, h_N}\rangle _s - \int_0^s h_N  f( x_v ) Q_{h_N}^2 ( x_v- x^*  ) dv | \right) = 0,
\]
for all $ s \geq t_0, $ 
such that the angle bracket process of  $(W^{N, h_N}_s)_{ s \geq t_0} $ converges in probability to 
\begin{multline*}
\lim_{N \to \infty } \int_0^s f(x_v)\frac{1}{h_N}  Q^2 ( \frac{x_v- x^*}{h_N}) dv
  = \lim_{N \to \infty } \int_{\frac{x_0-x^*}{h_N}}^{\frac{x_s- x^*}{h_N}} f ( x^* + h_N u ) Q^2 ( u)  \frac{1}{F ( x^* + h_N u ) }  du \\
 = \frac{f(x^*)}{|F(x^*)|} \int_{ - \infty }^{\infty} Q^2 (u) du.
 \end{multline*}
To obtain the last line above we have used the same arguments as those in the proof of Proposition~\ref{prop:ANh}. 

Therefore,  the condition $[\gamma'_5-D]$ of  the martingale convergence theorem Theorem VIII.3.11 of \cite{JSh} is satisfied and implies the weak convergence, in Skorokhod space, of $( W^{N,h_N}_s)_{ s \geq t_0 } $ to a continuous Gaussian limit martingale $ (W_s)_{s \geq t_0 } $ which has constant quadratic variation $ C_s \coloneqq C\coloneqq  \frac{f(x^*)}{|F(x^*)|} \int_{ - \infty }^{\infty} Q^2 (u) du $ for all $ s \geq t_0.$ In other words, almost surely, $ W_s = W_{t_0} $ for all $ s \geq t_0, $ where $ W_{t_0} \sim {\mathcal N} ( 0, C).$ In particular, taking $s= T,$ this implies the first statement of the theorem. The second statement follows immediately since
\[
 Nh_N\E_{x_0}^f [(M_T^{N,h_N})^2] =\E_{x_0}^f [  \langle W^{N, h_N}\rangle_T]
\]
and since all results that have been obtained are uniform with respect to the function class $ H ( \beta,l, L ) .$
$\qed$

\subsection{Proof of Theorem~\ref{theo:main}, item~\ref{itemb:theomain}) } This statement now follows from decomposition \eqref{eq:biais-variance},  Theorem~\ref{theo:martingale}, 
Theorem~\ref {theo:biais} and  Proposition~\ref {prop:ANh}~\eqref{item2:propanh}. $\qed$

\subsection{Proof of Theorem~\ref{theo:main}, item~\ref{itema:theomain}) } 

Remember that
\[
\Omega_T^{N, h_N, \eps}\coloneqq\{|\frac {A_T^{N,h_N}}{A_T^{h_N}}-1| \le \eps\}.
\]
We have already proved in item~\eqref{item3:propanh} of Proposition~\ref{prop:ANh} that $ \Pr_{x_0}^f ( (\Omega_T^{N, h_N, \eps})^c ) \to 0$ as $ N \to \infty, $ uniformly in $f \in H( \beta, l , L).$ 
Using the decomposition \eqref{eq:biais-variance}, we have that
\begin{multline}
\E \left[|\hat f_T^{N,h}(x^*)-f(x^*)|^2\indiq_{\Omega_T^{N, h_N, \eps}}\right]\leq 
\E\left[\left(\frac {M_T^{N,h_N}}{A_T^{N,h_N}}\right)^2\indiq_{\Omega_T^{N,h_N,\eps}}\right]+\E\left[\left(\frac{B_T^{N,h_N}}{A_T^{N,h_N}}\right)^2\indiq_{\Omega_T^{N, h_N, \eps}} \right]\\\leq
(1 -  \eps)^{-2}\left\{(A_T^{h_N})^{-2}\E\left[\left({M_T^{N,h_N}}{}\right)^2\right]+(A_T^{h_N})^{-2}\E\left[\left({B_T^{N,h_N}}{}\right)^2 \right]\right\}
.
\end{multline}
Using Proposition~\ref{prop:ANh} item~\eqref{item2:propanh}, $A_T^{h_N}$ converges to $ 1/ |F ( x^* )| $ which is finite since $ F (x^* ) \neq 0.$ 
Using moreover Theorems \ref{theo:martingale} and \ref{theo:biais}, we have
\[
 \limsup_{N \to \infty}   \sup_{f \in H( \beta, l, L)} N h_N \E_{x_0}^f [| M_T^{N, h_N}|^2] < \infty.  
\]
Finally,  if $ h_N =  {\mathcal O} ( N^{ - \frac{1}{ 2 \beta + 1 } }) ,$ then 
\[
 \limsup_{N \to \infty}   \sup_{f \in H( \beta,l, L)}{N h_N } \E_{x_0}^f | B_T^{N, h_N} |^2 < \infty .
\]
This concludes the proof. $\qed$

\section{Simulations}\label{sec:simu}

The aim of this section is to simulate system \eqref{eq:finitesystem} and to compute the estimator \eqref{eq:kernelestimator}
for some points $x^*\in \R$. 
We will consider systems with $N=20000$ neurons and we will consider $h=h_N=N^{-0.49} \approx 0.0078$. With this choice of $h_N$, we have $h_N^2 N \to \infty$ as $N\to \infty$, as assumed in Theorem \ref{theo:main}. Each system is simulated $10$ times, and for each simulation the initial membrane potential of the neurons will be the same. We simulate until time $T=10, $ and this is the time at which the estimator $\hat{f}^{N, h }_T$ is computed.

We consider the uniform rectangular kernel 
\[
Q(u)=\frac{1}{2}\mathbf{1}_{\{-1\leq u \leq 1\}},
\]
for any $u\in \R$.
Even though this kernel does not satisfy our assumptions, with this kernel the denominator of \eqref{eq:kernelestimator} has a simple solution: it is exactly the occupation time of  $]x^* - h_N, x^* + h_N [$ during $[0, T ]$.

In our examples, between jumps, the jump rate decreases over time. Therefore,  we can use an acceptance-rejection procedure, informally described by Algorithm \ref{algorithm}.

\begin{algorithm}[ht]
\caption{Simulation of system with an acceptance-rejection procedure}
\label{algorithm}
\begin{algorithmic}[1]
\State Inputs: $f,b,\nu,x_0,N$ and $T$.
\State Initialize the vector $X_0^{N,i} \leftarrow x_0$, for all $i \in I$.
\State Set $t \leftarrow  0$ (the simulation time).

\While{$t < T$}

        \State Draw $E$ (the candidate jump time) with exponential distribution having parameter
        \(
        \sum_{i \in I} f(X_t^{N,i}).
        \)
        
        \State Draw $A$ (accept the jump) with Bernoulli distribution with probability of success
        \[
        \frac{\sum_{i \in I} f (X_t^{N,i}+\int_{t}^{t+E} b ( X_s^{N, i } ) ds)}{\sum_{i \in I} f(X_t^{N,i})}.
        \]

        \If{$A=0$}
       \State Set
        \[
        X_{t+r}^{N, i} =  X_{t}^{N, i }+\int_{t}^{t+r} b ( X_s^{N, i } ) ds, \quad \text{ for } r\in (0,E] \text{ and } i\in I.
        \]

        \Else

            \State Choose $j$ (the index of the spiking neuron) with probability
    \[
    \frac{ f\left(X_t^{N,j}+\int_{t}^{t+E} b ( X_s^{N, j } ) ds\right)}{\sum_{i \in I} f (X_t^{N,i}+\int_{t}^{t+E} b ( X_s^{N, i } ) ds)}, \quad \quad  \text{ for} \quad  j \in I.
    \]

    \State Draw $U \sim\nu$ (the height of the jump).

     \State Set
        \[
        X_{t+r}^{N, i} =  X_{t}^{N, i }+\int_{t}^{t+r} b ( X_s^{N, i } ) ds+U \times  \mathbf{1}_{\{i\neq j, r=E\}}, \quad \text{ for } r\in (0,E] \text{ and } i \in I.
        \]

        \EndIf

        \State Set $t\leftarrow t + E$.
        
\EndWhile
\end{algorithmic}
\end{algorithm}

Let us first consider the system with jump rate function 
\[
f(x)=2-e^{-x^2}
\]
and drift function 
\[
b(x)=-x.
\]
Note that with the combination of a drift function pushing the membrane potential to $0$ and a jump rate $f(x)$ which is an increasing function of $|x|$, we can simulate the system using Algorithm \ref{algorithm}. We use the fact that the system follows an exponential flow between jumps to obtain occupation times and compute the integral in the denominator of \eqref{eq:kernelestimator}.

We consider $\nu \sim Uniform(-2,3)$ such that $w=0.5$. Note that the membrane potential assumes values in $\R$.
The unique solution of
\[
F(x)=b(x)+wf(x)=0
\]
is $x\approx 0.6889$.
The initial membrane potential of the system is set to $x_0=-1$. The results of the simulations are summarized in Figure \ref{fig:simulation1}. These results indicate that the Mean Squared Error decreases as $x^*, $ the position where we want to estimate,  increases and gets closer to $0.6889$. This is consistent with item b) of Theorem \ref{theo:main}, since the estimation variance is a factor of $F (x^*) $ and thus decreases as $ x^* $ approaches the equilibrium of the limit flow.

\begin{figure}[htbp] 
\center{\includegraphics[width=0.7\textwidth]
         {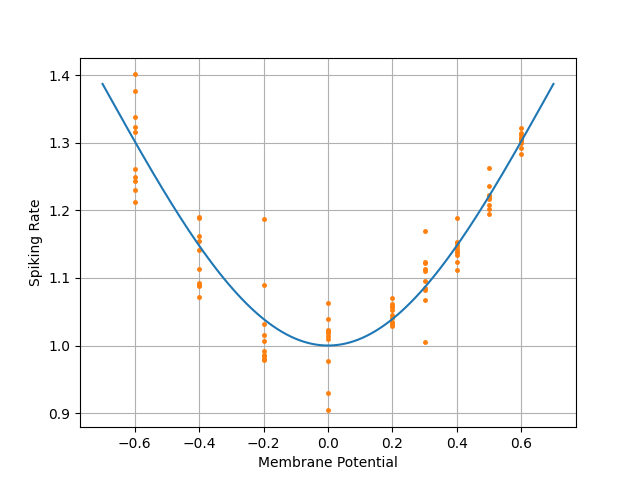}}
\caption{The blue curve in the figure is the spiking rate function $f$. The orange points in the figure are the values of the estimator calculated during each one of the $10$ simulations. The points in which the estimator was calculated are
$-0.6,-0.4, -0.2, 0, 0.2, 0.3, 0.4, 0.5$ and $0.6$. The successive Mean Squared Error for these points are approximately the following, respectively: 
$0.0038, 0.0021, 0.0041, 0.0022, 0.0003, 0.0018, 0.0004, 0.0003,$ and $0.0001$.
}
\label{fig:simulation1}
\end{figure}

\subsection{Partial observation of the system}

Let us now consider the estimator based on partial observations of the system given by \eqref{eq:kernelestimator1small}. Recall that we assume that we have access to the trajectories of $\gamma_N$ neurons, where $ \gamma_N \in \N$ such that $ \gamma_N  \to \infty $ as $ N \to \infty, $ $ \gamma_N \le N.$ Considering a set $\mathcal{S}\subset \{1,\ldots,N\}$ such that $|\mathcal{S}|=\gamma_N$, the estimator then reads as 
\begin{equation*}
\hat f_T^{N, \gamma_N, h } ( x^*) = \frac{ \sum_{i\in \mathcal{S}  } \int_{[0,T]\times \R_+}   Q_h ( X_{s-}^{N, i }-x^*) \indiq_{\{ z \le f ( X_{s-}^{N, i } ) \}}  \bar \pi^{i} (ds, dz)}{ \int_0^T \sum_{i\in \mathcal{S}}  Q_h ( X_s^{N, i} - x^*) ds }.
\end{equation*}

Considering the same system and the same simulations as above, let us compute the estimator given by \eqref{eq:kernelestimator1small} successively for
\[
\gamma_N\in\left\{\frac{N}{2}, \frac{N}{4}, \frac{N}{20}, \frac{N}{200}\right\},
\]
i.e, $\gamma_N\in \{10000, 5000, 1000, 100\}.$
We keep $h_N=N^{-0.49}$ such that $h_N^2 \gamma_N\to\infty$, as $N\to\infty$. We consider the evolution of the neurons from $1$ to $10000$ to compute the estimator with $\gamma_N=10000$, we consider the evolution of the neurons from $10.001$ to $15.000$ to compute the estimator with $\gamma_N=5000$, and so on such that no neuron is used in the computation of two or more partial observation estimators. The results of the simulations are summarized in Figure \ref{fig:simulation2} and Table \ref{table:MSE}. These results illustrate that the partial observation estimator gives estimates with a greater variance compared to the estimator \eqref{eq:kernelestimator}. While the cases $\gamma_N=10000$ and $\gamma_N=5000$  give a Mean Squared Error (MSE) comparable to the one obtained by \eqref{eq:kernelestimator}, the case $\gamma_N=1000$ gives a good MSE only for points sufficiently close to $0.6889$ and the case $\gamma_N=100$ gives a reasonable MSE only for the point $0.6$.
These results show again that the estimation error decreases as the estimated point gets closer to the equilibrium point $0.6889$. 

\begin{figure}[htbp] 
\centering
\begin{subfigure}[b]{0.49\textwidth}
    \centering
    \includegraphics[width=\textwidth]{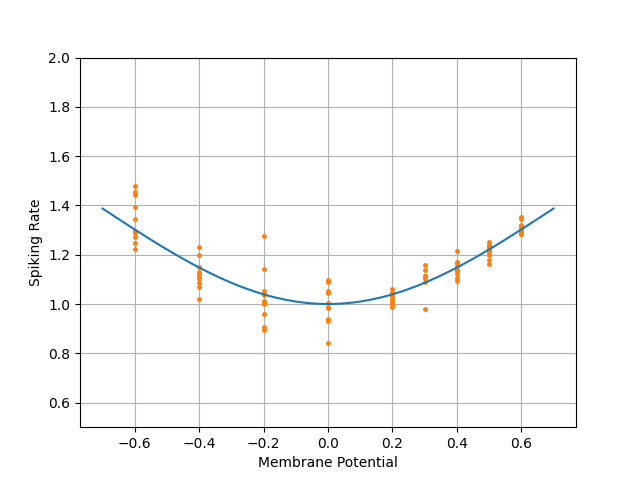}
    \caption{Partial observation estimator with $\gamma_N=10000$.}
\end{subfigure}
\begin{subfigure}[b]{0.49\textwidth}
    \centering
    \includegraphics[width=\textwidth]{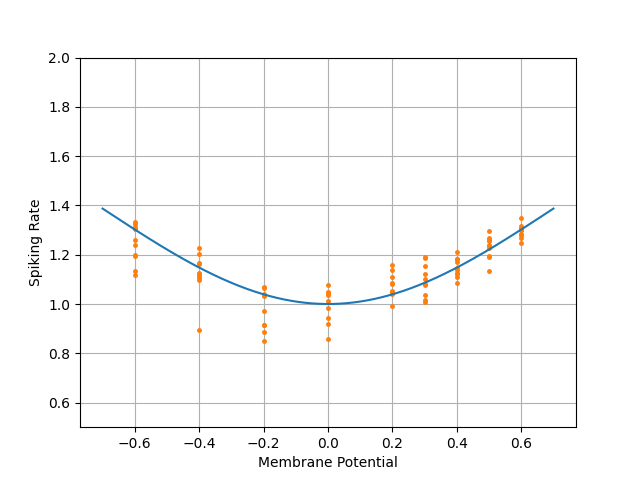}
    \caption{Partial observation estimator with $\gamma_N=5000$.}
\end{subfigure}
\begin{subfigure}[b]{0.49\textwidth}
    \centering
    \includegraphics[width=\textwidth]{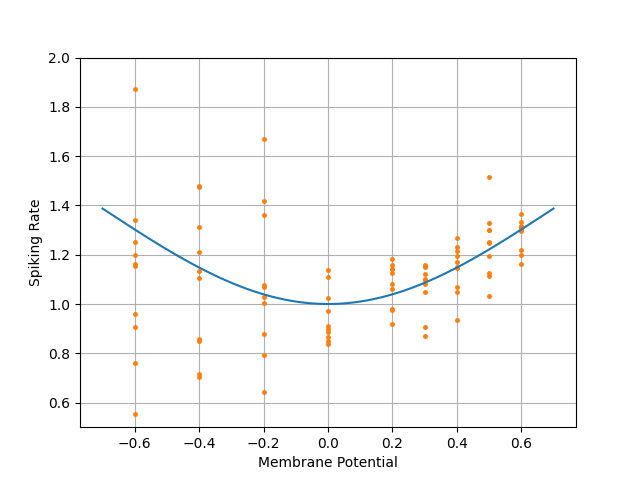}
    \caption{Partial observation estimator with $\gamma_N=1000$.}
\end{subfigure}
\begin{subfigure}[b]{0.49\textwidth}
    \centering
    \includegraphics[width=\textwidth]{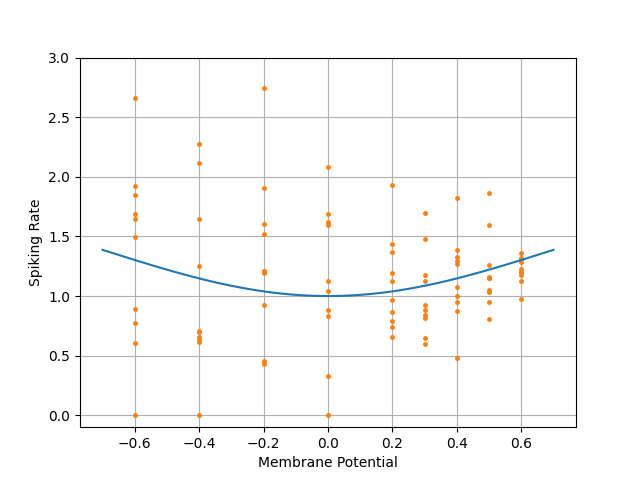}
    \caption{Partial observation estimator with $\gamma_N=100$.}
\end{subfigure}
\caption{The blue curve in each sub-figure is the spiking rate function $f$. The orange points in each sub-figure are the values of the partial observation  estimator calculated during  each one of the $10$ simulations. The points in which the estimator was calculated are
$-0.6,-0.4, -0.2, -0, 0.2, 0.3, 0.4, 0.5$ and $0.6$. The Mean Squared Errors associated to these simulations are presented in Table \ref{table:MSE}.}
\label{fig:simulation2}
\end{figure}

\begin{table}[!htb]
\resizebox{\columnwidth}{!}{%
\begin{tabular}{|ll|l|l|l|l|l|l|l|l|l|}
\hline
\multicolumn{2}{|l|}{Points}                    & -0.6 & -0.4 &  -0.2 &  0 &  0.2 &  0.3 & 0.4 & 0.5 &  0.6 \\ \hline
\multicolumn{1}{|l|}{\multirow{5}{*}{MSE}}   & All neurons & 0.0038 & 0.0021 & 0.0041 & 0.0022 & 0.0003 & 0.0018 & 0.0004 & 0.0003 & 0.0001\\ \cline{2-11} 
\multicolumn{1}{|l|}{}                       & $\gamma_N=10000$ &  0.0095 & 0.0041 & 0.0115 & 0.0056 & 0.0008 & 0.0023 & 0.0011 & 0.0008 & 0.0006 \\ \cline{2-11} 
\multicolumn{1}{|l|}{}                       & $\gamma_N=5000$ & 0.0093 & 0.0081 & 0.0095 & 0.0043 & 0.0035 & 0.0040 & 0.0014 & 0.0021 & 0.0008 \\ \cline{2-11} 
\multicolumn{1}{|l|}{}                       & $\gamma_N=1000$ & 0.1509 & 0.0809 & 0.0889 & 0.0131 & 0.0087 & 0.0095 & 0.0090 & 0.0169 & 0.0042  \\ \cline{2-11} 
\multicolumn{1}{|l|}{}                       & $\gamma_N=100$ & 0.5460 & 0.4978 & 0.5024 & 0.3888 & 0.1422 & 0.1169 & 0.1173 & 0.0878 & 0.0183   \\ \hline
\end{tabular}
}
\caption{Mean Squared Errors (MSE) of the simulations associated with Figure \ref{fig:simulation2}.}
\label{table:MSE}
\end{table}

\subsection{Systems with a finite number of spikes}

Let us consider the system \eqref{eq:finitesystem} in the purely excitatory regime when all membrane potentials are non-negative and $ \nu $ is supported by $ \R_+ ,$ with jump rate function 
\[
f(x)=\log(1+x)
\]
and drift function 
\[
b(x)=-x.
\]
In this situation, we can still simulate the system using an acceptance-rejection algorithm. This system is similar to the one studied by \cite{do} and \cite{Touboul}. The only difference is that in the system considered in these articles, each neuron resets its membrane potential to $0$ at each spiking time. Moreover, \cite{do} considers the case $\nu=\delta_h,$ for $h>0$. These articles prove that if $ \int u  \nu ( du) <\infty,$ then the system stops spiking after a finite amount of time almost surely. In the case $\nu=\delta_h,$ if $h$ is sufficiently small, the system quickly stops spiking. However, for sufficiently high values of the synaptic weight $h$,  \cite{lochandmonm} proves that the system exhibits a metastable behavior, meaning that the time at which the system stops spiking is a large and unpredictable random time.

Although our system does not include the reset to $0$ of the spiking neuron, also in the present situation the following result holds true. 

\begin{proposition}
For the system \eqref{eq:finitesystem} with $ f (x) = \log ( 1 + x ), b (x) = - x $ and $ \nu ( \R_-) = 0, $ we have that 
\[ \Pr_{x_0}^f ( \sum_{k=1}^\infty \indiq_{ \{ T_k < \infty \}} < \infty ) = 1.\]
\end{proposition} 

\begin{proof}
Fix some $ u  = (u^1, \ldots, u^N ) \in \R_+^N $ and write $ \Pr_u^f $ for the probability under which $ X^{N, i }_0 = u^i $ for each $ 1 \le i \le N.$
Following the arguments of Proposition 3.1 of \cite{do}, 
for any $r>0$ and any $u\in \R_+^N$ such that $u^i<r$, $1\leq i\leq N$, it holds, for all $0<t<T_1$, that
$ X^{N,i}_t  \leq re^{-t}$. Therefore
\[
\Pr_u^f(T_1>t)\geq \exp\Big\{-N\int^t_0 \log(1+re^{-s})ds\Big\}=\exp\Big\{-N\int^{r}_{re^{-t}}\frac{\log(1+u)}{u}du\Big\}.
\]
Taking $t$ to infinity, for any $r>0$, we have
\[
\inf_{u:u^i \leq r}\Pr_u^f(T_1=\infty) \geq \exp\Big\{-N\int^{r}_{0}\frac{\log(1+u)}{u}du\Big\}>0.
\]

Moreover, the function $ V (x) = \sum_{i=1}^N |x^i | $ is a Lyapunov function of the finite system, that is, there exists a constant $ C> 0 $ such that for all $ x \in \R_+^N , $ 
\[ A^N V (x) \le \frac12 V (x)  + C ,\]
where 
\[ A^N V ( x) = \sum_{ i= 1 }^N (- x^i ) \frac{ \partial V}{\partial x^i } (x) + \sum_{i=1}^N f ( x^i ) \int_{\R_+}[V ( x + \frac{u}{N}  \boldsymbol{1} - \frac{u}{N} e_i ) - V ( x)]  \nu ( du ) \]
is the generator of \eqref{eq:finitesystem}. 
Here $ \boldsymbol{1}$ the all $1$ vector in $ \R^N $ and $e_i$ the $i-$th unit vector.  

Then the proof of Theorem 2.3 in \cite{do} can be applied to our situation, and the result follows. 
\end{proof}

\begin{remark}
Presumably, the strategy of \cite{lochandmonm} could be adapted to the present framework to prove that for large values of $ w, $ metastability holds. In this case, the limit system possesses (at least) two equilibrium points, one strictly positive representing some sustained activity of the system, the other one being equal to $0, $ representing the silent system. To prove this metastability result is outside the scope of the present article. 
\end{remark}

In the following simulations, we use 
\[
f(x)=\log(1+x)
\]
and 
\[
b(x)=-x.
\]

We first set  $\nu \sim Uniform(0,4)$ and therefore, $w=2$. This system displays the metastable behavior described by \cite{lochandmonm}. 
The two solutions of
\[
F(x)=b(x)+wf(x)=0
\]
are $x=0$ and $x\approx 2.5129$. The initial membrane potential of the system is $x_0=0.1$. The results of the simulation are summarized in Figure \ref{fig:simulation3}. 

\begin{figure}[htbp] 
\center{\includegraphics[width=0.7\textwidth]
         {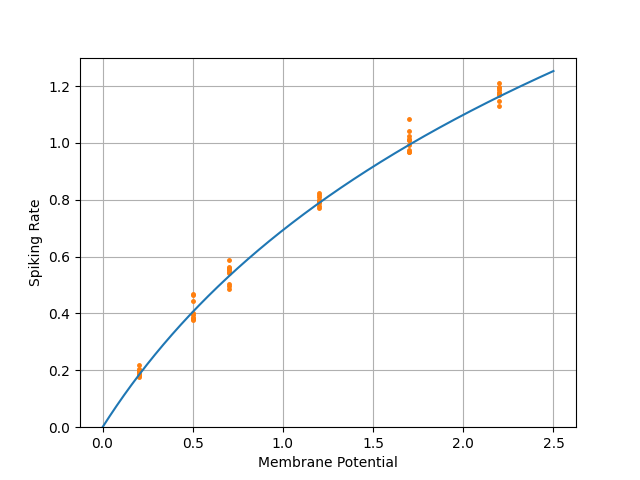}}
\caption{The blue curve in the figure is the spiking rate function $f$. The orange points in the figure are the values of the estimator calculated during each one of the $10$ simulations. The points in which the estimator was calculated are $0.2, 0.5, 0.7, 1.2, 1.7$ and $2.2$. The successive Mean Squared Error for these points are approximately the following: $0.0003 , 0.0012 , 0.0011 , 0.0003 , 0.0014$ and $0.0007$.
}
\label{fig:simulation3}
\end{figure}

Now, let us consider a new system with the same jump rate function
$
f(x)=\log(1+x),
$
the same drift function 
$
b(x)=-x,
$
but with a new probability measure $\nu$, given by $\nu \sim Uniform(0,1)$. Note that $w=0.5$. In this system, the membrane potential of neurons quickly approaches zero, leading to extinction of the system as described by \cite{do}. 
The unique non-negative solution of
\[
b(x)+wf(x)=0
\]
is $x=0$. The initial membrane potential of the system is $x_0=1$.
The results of the simulation are summarized in Figure \ref{fig:simulation4}. 

The results displayed in Figures \ref{fig:simulation3} and \ref{fig:simulation4} show that the estimator \eqref{eq:kernelestimator} gives good estimates of the jump rate function even in the case in which we consider systems that stop spiking almost surely.

\begin{figure}[htbp] 
\center{\includegraphics[width=0.7\textwidth]
         {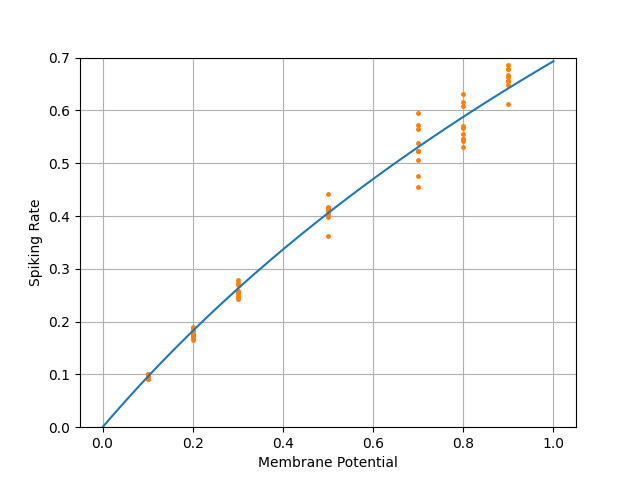}}
\caption{The blue curve in the figure is the spiking rate function $f$. The orange points in the figure are the values of the estimator calculated during each one of the $10$ simulations. The points in which the estimator was calculated are $0.1, 0.2, 0.3, 0.5, 0.7, 0.8,$ and $0.9$. The successive Mean Squared Error for these points are approximately the following: 
$0.00001, 0.00009, 0.00015, 0.00035, 0.00169, 0.00135,$ and $0.00071$.
}
\label{fig:simulation4}
\end{figure}

\begin{remark}[Some remarks on systems that stop spiking almost surely]
The simulation of Figure \ref{fig:simulation4} shows a system that stops spiking almost surely after a finite time and that has been initially perturbed such that we observe it out of its equilibrium state (the silent state). We interpret this initial perturbation as the influence of some external stimulus to which the system has been exposed. The influence of the stimulus induces a high value of the initial potential of the neurons. 
Then this stimulus is switched off, and we observe the system during some time. If the initial stimulus is sufficiently large, then the influence of this stimulus survives within the system during some macroscopic time such that we are able to perform our estimation procedure.  
\end{remark}

\section{Appendix}

\subsection{Proof of Theorem \ref{theo:strongapprox}}
Recall that $\tilde \pi^i  ( dt, dz, du ) = \pi^i ( dt, dz, du ) - dt dz \nu ( du) ,$  $ 1 \le i \le N, $ denote the compensated Poisson random measures. Recall also that $ w = \int u \nu (du)  .$ 

Clearly, for any $ 1 \le i \le N, $ 
\begin{equation}\label{eq:Pt0}
X^{N, i }_t = x_0 + \int_0^t b( X_s^{N, i } ) ds + \frac{w}{N} \sum_{ j=1}^N\int_0^t f ( X_s^{N, j } ) ds 
+ \frac 1{\sqrt N}L_t^{N} 
-E^{N,i}_t ,
\end{equation}
where
\[E^{N,i}_t \coloneqq\frac 1N\int_{[0, t]}\int_{\R_+} \int_{\R} u \indiq_{ z \le f ( X_{s-}^{N, i } ) }  \pi^i (ds, dz, du )\]
and
\[L_t^{N}\coloneqq\frac1{\sqrt N }\sum_{ j = 1  }^N \int_{[0, t]}\int_{\R_+} \int_{\R} u \indiq_{ z \le f ( X_{s-}^{N, j } ) }\tilde  \pi^j (ds, dz, du ) .\]
Recall that $F=b+wf$ and that $ x_t = x_0 + \int_0^t F( x_s) ds.$ 
Therefore
\begin{equation}
X_t^{N, i } - x_t =\int_0^t (F( X_s^{N, i } )-F(x_s)) ds+G^N_t+\frac 1{\sqrt N}L^N_t - E^{N,i}_t,
\end{equation}
where 
\[G^N_t=\frac wN\sum_{j=1}^N\int_0^t (f(X_s^{N,j})-f(X_s^{N,i}))ds.\]

Using the Gronwall lemma pathwise (that is, for any fixed $ \omega$), we obtain for any $ t \le T,$ 
\begin{equation}\label{eq:bV}
 \sup_{t\leq T} | X_t^{N, i } - x_t |  \le e^{  \|F\|_{Lip} T} \frac 1{\sqrt N}\left[ \sqrt N \sup_{t\leq T}|G^N_t|+\sup_{t\leq T}|L^N_t|+\sqrt N \sup_{t\leq T}|E^{N,i}_t |  \right] .
\end{equation}

As
\[\sup_{t\leq T}|E_t^{N,i}|\leq \frac 1N\int_{[0,T]}\int_{\R_+}\int_{\R}|u|\indiq_{ \{ z \le L \}} \pi^i (ds,dz,du),\]
we deduce that
\begin{equation}\label{eq:bE}
\sup_{ f \in H ( \beta, l ,L) } \E_{x_0}^f [\sup_{t\leq T}|E_t^{N,i}|^p]\leq \frac 1{N^p}C_{T}(p).
\end{equation}

Moreover, for any fixed $i \neq j , $ 
\[ X_t^{N, i } - X_t^{N, j } = \int_0^t (b( X_s^{N, i } ) - b ( X_s^{N, j } ) )ds -E_t^{N,i}+E_t^{N,j} ,
\]
such that, due to the Lipschitz continuity of $b$ and applying the pathwise Gronwall lemma once more,  
\[ \sup_{ t \le T } | X^{N, i }_t - X^{N, j }_t | \le e^{ CT}(\sup_{ t \le T }|E_t^{N,i}|+\sup_{ t \le T }|E_t^{N,j}|). 
\]

Recall that by definition of $ H(\beta, l , L), $ $ \|f\|_{Lip} \le L.$ So 
\[|G^N_t|\leq \frac {L |w|}N\sum_{j=1}^N\int_0^t |X_s^{N,j}-X_s^{N,i}|ds,\]
which implies that
\[\sup_{t\leq T}|G_t^N|^p\leq C_{T} (p) \frac 1N\sum_{j=1}^N (\sup_{t\leq T}|E_t^{N,i}|^p+\sup_{t\leq T}|E_t^{N,j}|^p)\]
and finally 
\begin{equation}\label{eq:bG}
\sup_{ f \in H ( \beta, l ,L) } \E_{x_0}^f\sup_{t\leq T}|G_t^N|^p\leq C_{T} (p) \frac 1{N^p}.
\end{equation}
We then prove that $ \sup_{ f \in H( \beta, l, L) } \E_{x_0}^f \sup_{ t \le T}| L^N_t|^p $ $< \infty $ for any $ p \geq 1,$ which is sufficient to conclude.

By H\"older's inequality, it suffices to consider the case $ p \geq 2.$ By the Burkh\"older-Davis-Gundy inequality of power $p,$ we have
\begin{multline*}
\E_{x_0}^f  \sup_{ t \le T}| L^N_t|^p\leq C (p) \E([L^N_t,L^N_t]^{p/2}) \\
=C (p) \E_{x_0}^f  \left[\left( \frac {1 }{N} \sum_{j=1}^N \int_{[0, t ]} \int_{\R_+} \int_{\R} u^2 \indiq_{\{ z \le f ( X_{s-}^{N, j } )\} }  \pi^j (ds, dz, du ) \right)^{p/2}\right].
\end{multline*}
Since $ x \mapsto |x|^{p/2} $ is convex, using Jensen's inequality for the discrete measure $ \frac1N \sum_{ j= 1}^N, $ we upper bound the above expression by 
\begin{multline*}
    \E_{x_0}^f  \left[ \frac {1 }{N} \sum_{j=1}^N \left( \int_{ [0, t ] } \int_{\R_+} \int_{\R} u^2   \indiq_{\{ z \le f ( X_{s-}^{N, j } )\} }  \pi^j (ds, dz, du ) \right)^{p/2} \right] \\
    =   \E_{x_0}^f \left( \int_{ [0, t ] } \int_{\R_+} \int_{\R} u^2   \indiq_{\{ z \le f ( X_{s-}^{N, 1 } )\} }  \pi^1 (ds, dz, du ) \right)^{p/2} ,
\end{multline*}    
where we have used the exchangeability of the system. Since $f$ is bounded by $L,$ we also have that 
\[ \int_{ [0, t ] } \int_{\R_+} \int_{\R} u^2   \indiq_{\{ z \le f ( X_{s-}^{N, 1 } )\} }  \pi^1 (ds, dz, du ) \le \int_{ [0, t ] } \int_{\R_+} \int_{\R} u^2   \indiq_{\{ z \le L\}  }  \pi^1 (ds, dz, du ),\]
which is a compound Poisson process with jumps possessing all moments by our assumptions on $ \nu (du).$ As a consequence, it also possesses all moments which shows that
\begin{equation}\label{eq:bL}
\sup_{ f \in H ( \beta, l ,L) } \E_{x_0}^f \sup_{ t \le T}| L^N_t|^p\leq C_T (p).
\end{equation}
Finally from \eqref{eq:bV},\eqref{eq:bE},\eqref{eq:bG} and \eqref{eq:bL} we obtain
\[\sup_{ f \in H ( \beta, l ,L) } \E_{x_0}^f\sup_{t\leq T}|X_t^{N,i}-x_t|^p\leq C_{T} (p) N^{-p/2},\]
which finishes the proof.
$ \qed $



%
\providecommand{\bysame}{\leavevmode\hbox to3em{\hrulefill}\thinspace}
\providecommand{\MR}{\relax\ifhmode\unskip\space\fi MR }
\providecommand{\MRhref}[2]{%
  \href{http://www.ams.org/mathscinet-getitem?mr=#1}{#2}
}
\providecommand{\bysame}{\leavevmode\hbox to3em{\hrulefill}\thinspace}


\ACKNO{
This work was produced as part of the activities of FAPESP  Research, Innovation and Dissemination Center for Neuromathematics (grant \# 2013/ 07699-0 , S.Paulo Research Foundation (FAPESP)) and as part of  the ANR project ANR-19-CE40-0024.
KL was successively supported by FAPESP fellowship (grant 2022/07386-0 and 2023/12335-9).

We thank two anonymous reviewers for their careful reading and for many interesting questions and remarks that helped us a lot to improve a first version of the manuscript.

We warmly thank Etienne Tanr\'e for his careful reading and his help with editing issues.}

\nocite{*}
\end{document}